\documentclass[11pt]{amsart}

\usepackage{amsmath,amsthm,enumitem,amssymb,graphicx,color,bm, mathabx}
\usepackage[margin=2cm]{caption}

\usepackage[colorlinks,citecolor=cyan,linkcolor=magenta
]{hyperref}
\usepackage[nameinlink]{cleveref}
\usepackage[title]{appendix}

\usepackage{geometry}
\newcommand{\mr}{\mathring}
\newcommand{\wt}{\widetilde}

\newcommand{\Z}{\mathbb{Z}}

\newcommand{\R}{\mathbb{R}}

\newcommand{\C}{\mathcal{C}}

\newcommand{\mc}{\mathcal}
\newcommand{\hbs}{\tau^{(2)}}
\newcommand*\hex{{\raisebox{-.15ex}{\text{\includegraphics[width=.9em]{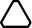}}}}}
\newcommand*\smhex{{\raisebox{-.15ex}{\text{\includegraphics[width=.6em]{char_customhex.pdf}}}}}
\newcommand*\tet{{\raisebox{-.3ex}{\text{\includegraphics[width=.9em]{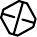}}}}}
\newcommand*\taut{{\raisebox{-.3ex}{\text{\includegraphics[width=.9em]{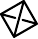}}}}}

\newcommand{\tri}{\triangle}
\newcommand{\del}{\partial}

\DeclareMathOperator{\closure}{cl}
\DeclareMathOperator{\cone}{cone}
\DeclareMathOperator{\intr}{int}

\DeclareMathOperator{\ind}{index}

\DeclareMathOperator{\brloc}{brloc}

\DeclareMathOperator{\sgn}{sign}

\DeclareMathOperator{\core}{core}
\DeclareMathOperator{\coll}{coll}
\DeclareMathOperator{\Neg}{neg}
\DeclareMathOperator{\shad}{shadow}
\DeclareMathOperator{\prongs}{prongs}

\renewcommand{\phi}{\varphi}

\newtheorem{thm}{Theorem}

\numberwithin{equation}{section}
\newtheorem{theorem}[equation]{Theorem}

\newtheorem{lemma}[equation]{Lemma}
\newtheorem{corollary}[equation]{Corollary}

\newtheorem*{maintheorem}{Theorems A and B}

\theoremstyle{definition}
\newtheorem{definition}[equation]{Definition}

\newtheorem*{question*}{Question}

\theoremstyle{remark}
\newtheorem{remark}[equation]{Remark}
\newtheorem{example}[]{Example}

\begin{document}

\title[Veering triangulations and the Thurston norm]{Veering triangulations and the Thurston norm: homology to isotopy}
\author{Michael P. Landry}
\address{1 Brookings Drive\\
Washington University in Saint Louis\\
Saint Louis, MO 63130 }
\email{\href{mailto:mlandry@wustl.edu}{mlandry@wustl.edu}}

\date{}
\maketitle
\begin{abstract}
We show that a veering triangulation $\tau$ specifies a face $\sigma$ of the Thurston norm ball of a closed 3-manifold, and computes the Thurston norm in the cone over $\sigma$. Further, we show that $\tau$ collates exactly the taut surfaces representing classes in the cone over $\sigma$ up to isotopy. The analysis includes nonlayered veering triangulations and nonfibered faces. We also prove an analogous theorem for manifolds with boundary that is integral to a theorem of Landry-Minsky-Taylor relating the Thurston norm to the {veering polynomial}, a new generalization of McMullen's Teichm\"uller polynomial.
\end{abstract}

%\setcounter{tocdepth}{2}
%\tableofcontents

\noindent

%%%%%%%%%%%%%%%%%%%%%%%%%%%%%%%%%%%%%%%%%%%%%%%%%%%%%%%%%%%%%%%%%%%%%%%%%%%%%%%%%%%%%%%%%%%%%%%%%%%%%%%%%%%%%%%%%%
\section{Introduction}

Let $M$ be a closed, oriented, irreducible 3-manifold. A \emph{taut surface} in $M$ is an embedded oriented surface which minimizes topological complexity in its homology class. Perhaps the simplest example is a fiber of a fibration of $M$ over the circle, if one exists. More generally, Thurston shows in \cite{Thu86} that any compact leaf of a taut foliation of $M$ is taut. The converse is also true: in \cite{Gab83}, Gabai shows that any taut surface is a compact leaf of some taut foliation of $M$.
The results herein show that \emph{veering triangulations}
%, called ``veering triangulations with a transverse taut structure" in some of the literature e.g. \cite{FutGue13}, 
are intimately related to taut surfaces and to the unit ball of the Thurston norm, an object of broad interest arising in areas from geometric group theory to Floer homology.

Before stating our results we give some broad-strokes definitions. A {veering triangulation} $\tau$ is a \emph{taut ideal triangulation} of a torally bounded 3-manifold $\mr M$ in the sense of Lackenby \cite{Lac00} satisfying an extra condition on $\tau \cap\partial \mr M$ (see \Cref{sec_veeringdef}). If $M$ is a Dehn filling of $\mr M$ we define a natural element $e_\tau\in H_1(M)$ called the \emph{Euler class of ${\tau}$}. The subset of $H_2(M)$ on which the Thurston norm $x$ agrees with the pairing $\langle -e_\tau,\cdot\rangle\colon H_2(M)\to \R$ is either $\{0\}$ or the nonnegative cone over some face, which we name ${\sigma_\tau}$, of the Thurston norm unit ball $B_x(M)$. We denote this cone by $\cone(\sigma_\tau)$. The 2-skeleton $\hbs$ of $\tau$ is a cooriented branched surface in $\mr M$ and an object we call a \emph{partial branched surface} when viewed as a subset of $M$. In \Cref{sec_pbs} we define these objects and say what it means for a partial branched surface to \emph{carry} a surface. The set of homology classes of closed curves in $\mr M$ positively transverse to $\hbs$ generates a convex polyhedral cone $\C_\tau\subset H_1(M)$ that we call the \emph{cone of homology directions of ${\tau}$}. This defines a dual cone $\C_\tau^\vee=\{\alpha\in H_2(M)\mid \langle \alpha, \gamma\rangle\ge 0 \text{ for all } \gamma\in \C_\tau\}$, where $\langle\cdot, \cdot\rangle$ denotes algebraic intersection. 

The following two theorems summarize our main results.

%Our first result says that any veering triangulation sees some face of the unit ball.

\begin{thm}\label{thma}
Let $\tau$ be a veering triangulation of a compact 3-manifold $\mr M$. If $M$ is obtained by Dehn filling each component of $\partial \mr M$ along slopes with $\ge 3$ prongs then $M$ is irreducible and atoroidal. Let $\sigma_\tau$ be the face of the Thurston norm ball $B_x(M)$ determined by the Euler class $e_\tau$. Then $\cone(\sigma_\tau)=\C_\tau^\vee$, and the codimension of $\sigma_\tau$ in $\partial B_x(M)$ is equal to the dimension of the largest linear subspace contained in $\C_\tau$.
\end{thm}
%Our second result, which is the main theorem of this paper, says that any veering triangulation sees a some face of the unit ball at the level of isotopy, not just homology.
\begin{thm}[Main theorem]\label{thmb}
Let $M$, $\tau$, $\sigma_\tau$ be as in \Cref{thma}. A surface $S\subset M$ is carried by $\hbs$ up to isotopy if and only if $S$ is taut and $[S]\in \cone(\sigma_\tau)$.
\end{thm}
In the theorem statements we allow for the possibility that $\sigma_\tau$ is the empty face, which we assign dimension $-1$. The term ``prongs" describes a function for each component of $\partial \mr M$ from $\{\text{boundary slopes}\}$ to $\Z_{\ge0}$, and is defined in \Cref{sec_flatprelim}. We restate these theorems at the end of \Cref{sec_negregion} and show how they follow from intermediate results in the paper.

%\subsection{Discussion}

\Cref{thma} demonstrates that if $\C_\tau\subsetneq H_1(M)$ then $\tau$ is linked to the combinatorics of $B_x(M)$ and to the Thurston norm itself. 
\Cref{thmb} says that $\hbs$ is a scaffolding off of which can be hung not only {some} taut representative of any homology class lying in $\cone(\sigma)$, but {every} taut representative of such a class up to isotopy. This is notable because while a fiber of a fibration $M\to S^1$ is the unique taut representative of its homology class up to isotopy (\cite[\S 3]{Thu86}), an integral class in $H_2(M)$ is not necessarily represented by a unique taut surface up to isotopy and the collection of taut surfaces representing a single homology class is not well understood. 
In fact it is a general phenomenon that one might wish to prove some statement about all taut representatives of a homology class and succeed only in proving the statement holds for a single taut representative of the homology class. We describe two examples of this in order to motivate the promotion of homology to isotopy.

\begin{example} 
The Fully Marked Surface Theorem \cite[Theorem 1.1]{GY20} of David Gabai and Mehdi Yazdi says that if $\mc F$ is a taut foliation of $M$ and $S$ is {fully marked} (meaning that its Euler characteristic is given by pairing with the Euler class of $\mc F$), then the tangent bundle of $\mc F$ is homotopic to the tangent bundle of a new taut foliation $\mc F'$ such that $S$ is homologous to a compact leaf $S'$ of $\mc F'$. Gabai and Yazdi conjecture \cite[Conjecture 1.5]{GY20} that it is not possible to replace ``homologous" with ``isotopic."
\end{example}

\begin{example}\label{ex_tst} 
Lee Mosher proves the Transverse Surface Theorem \cite[Theorem 1.3.1]{Mos92b} over the course of \cite{Mos89,Mos90,Mos91,Mos92b}, which says that if $\phi$ is a pseudo-Anosov flow on $M$, then any integral homology class in $H_2(M)$ which pairs nonnegatively with the homology class of each closed orbit of $\phi$ is represented by a taut surface which is \emph{almost transverse} (see \cite[\S 1]{Mos92b}) to $\phi$. It is not known if any taut surface representing such a homology class is almost transverse to $\phi$ up to isotopy, although Mosher was able to achieve a partial result in this direction \cite[\S 3]{Mos92b}.
\end{example}

\subsection{Related work on flows}\label{sec_relwork}

Unwritten work of Ian Agol and Fran\c{c}ois Gu\'eritaud shows that a pseudo-Anosov flow with {no perfect fits} (see e.g. \cite[Definition 2.2]{Fen12}) gives rise to a veering triangulation together with a Dehn filling, where the cores of the filling tori correspond to singular orbits and the combinatorial notion of {prongs} from \Cref{thma} corresponds to the number of prongs of a singular orbit. Work in progress of Saul Schleimer and Henry Segerman (\cite{SchSeg19} and more to come) aims to show the reverse, i.e. that the pseudo-Anosov flow can be reconstructed from the triangulation and the filling. Mosher has shown \cite[Flows Represent Faces]{Mos92b} that under certain conditions (in particular the case of no perfect fits), a pseudo-Anosov flow on $M$ represents a face of $B_x(M)$ in the sense that its cone of homology directions \cite[\S 1.2]{Mos92b} is dual to the cone on which its negative Euler class agrees with the Thurston norm. His Transverse Surface Theorem (see \Cref{ex_tst} above) then implies every integral class in this cone is represented by a surface almost transverse to the flow.

%Clearly many of the ideas in our Main Theorem are at play in the above paragraph. 
%Further, if one views a veering triangulation $\tau$ as a combinatorialization of a flow as above, then being carried by $\tau^{(2)}$ is analogous to almost transversality. 
Here is how our results fit into this picture:

\begin{itemize}
\item If one takes for granted the results in progress above, the statement of \Cref{thma} is unsurprising in light of Mosher's results. Its interest lies in the fact that the statement and its proof are entirely independent of pseudo-Anosov flows and the work of Mosher, Agol-Gu\'eritaud, and Schleimer-Segerman. Moreover the proof introduces the ideas of flattening and flat isotopy (see \Cref{sec_flatsurf}). These ideas are later applied to prove \Cref{thm_export}, which is key to a main result in \cite{LMT20} relating the Thurston norm to the veering polynomial, an object generalizing McMullen's Teichm\"uller polynomial \cite{McM00}.

\item Being carried by $\hbs$ is a combinatorial version of almost transversality. However, \Cref{thmb} is a much stronger result than could be obtained from simply viewing the above work through a combinatorial lens, even if one knew independently that being almost transverse to the flow implied being carried by $\hbs$ (this last fact is implied by our results). The key point is the distinction between homology and isotopy.
\end{itemize}

 \subsection{Veering triangulations}
Veering triangulations, which can be thought of as combinatorial versions of both foliations and flows, have enjoyed much recent interest. As these triangulations play a prominent role in the paper and are still perhaps esoteric to many readers, we survey some of the literature to date. Veering triangulations are introduced by Agol in \cite{Ago10} with the goal of analyzing the mapping tori of pseudo-Anosov homeomorphisms with small dilatation. The triangulations he studies and constructs are all \emph{layered}, meaning they are built from stacking tetrahedra on a surface and taking a quotient by some homeomorphism of the surface. In the literature on veering triangulations, few papers (\cite{HRST11,FutGue13,SchSeg19,LMT20} and this one) include analyses of the nonlayered case. Layered veering triangulations have been used to study Cannon-Thurston maps \cite{Gue15}, the curve and arc complexes of surfaces \cite{MinTay17, Str18}, and pseudo-Anosov flows on fibered hyperbolic 3-manifolds \cite{Lan18,Lan19}. There is another branch of study concerning the question of when a veering triangulation admits a geometric structure, or more generally a strict angle structure \cite{HRST11, FutGue13, HIS16, FTW20}. More recently there is the work of Agol-Gu\'eritaud and Schleimer-Segerman described in \Cref{sec_relwork}.

 \subsection{Applicability, methods, and outline}
 
The work of Agol-Gu\'eritaud described in \Cref{sec_relwork} shows that our results apply in any closed 3-manifold supporting a pseudo-Anosov flow with no perfect fits. By results of Calegari, any atoroidal 3-manifold admitting an $\R$-covered taut foliation \cite[Corollary 5.3.16]{Cal00} or more generally a taut foliation with 1-sided branching \cite[Corollary 4.2.9]{Cal01} admits a pseudo-Anosov flow transverse to the foliation. By results of Fenley \cite[Theorem G and Theorem H]{Fen12} these flows have no perfect fits. Hence our theorem can be applied in any atoroidal manifold admitting an $\R$-covered foliation or more generally a taut foliation with one-sided branching. The result also applies to fibered hyperbolic 3-manifolds using Agol's original construction; another way to see this is by noting that the suspension flow of a pseudo-Anosov map is a pseudo-Anosov flow with no perfect fits.
 
It is natural to ask to what extent Theorems \ref{thma} and \ref{thmb} can be used to compute $x$ and classify faces of $B_x(M)$. More specifically we can ask whether, given a taut surface in $M$, one can find a veering triangulation of a link exterior $\mr M$ in $M$ such that the filling $\mr M\to M$ satisfies the $\ge 3$ prongs condition. This remains unresolved for the time being but is the subject of work in progress by Chi Cheuk Tsang and the author.

%\subsection{Methods and outline}
We now describe the proofs of Theorems \ref{thma} and \ref{thmb}. The containment $\C_\tau^\vee\subset \cone(\sigma_\tau)$ in \Cref{thma} follows from \Cref{prop_carriedcriterion}, which is a combinatorial version of Mosher's Transverse Surface Theorem. The proof of \Cref{prop_carriedcriterion} relies on a result from \cite{LMT20} which we modify for our setting using techniques from our earlier paper \cite{Lan18}. The proof of the reverse containment is considerably more involved. In the course of the proof we take a taut surface $S$ representing a homology class in $\cone(\sigma_\tau)$ and systematically simplify the position of $S$ with respect to $\hbs$ via isotopies until it is carried by $\hbs$, proving the reverse containment and \Cref{thmb} simultaneously.
The process of simplification also requires keeping track of the position of $S$ with respect to two branched surfaces $B^s$ and $B^u$ associated to $\tau$, respectively called the \emph{stable} and \emph{unstable branched surfaces}. 
Our bookkeeping is made possible by a process we introduce called \emph{flattening} which takes $S$ and places it in a regular neighborhood of $\hbs$ by an isotopy, essentially allowing us to think of $S$ as carried by $\hbs$ with positive and negative weights. This endows $S$ with a combinatorial notion of area which we can minimize. Then we show that area-minimality implies $S$ is carried by $\hbs$.

We conclude the introduction by outlining the paper.
In \Cref{sec_bkgd} we give some background on Euler characteristic and train tracks, on branched surfaces and partial branched surfaces, and on the Thurston norm and dual Thurston norm. In \Cref{sec_veering} we define veering triangulations and develop some necessary combinatorics.
In \Cref{sec_flatsurf} we describe the flattening process and introduce a set of moves which allow us to pass between flattened surfaces by isotopy.
In \Cref{sec_easycontainment} we prove the easier containment $\C_\tau^\vee\subset \cone(\sigma_\tau)$.
In \Cref{sec_bigprops} we begin proving the reverse containment starting with a taut surface $S$ with $[S]\in \cone(\sigma_\tau)$ by first showing we can isotope  $S$ to have a nice position with respect to $B^s$ and $B^u$, in which we say $S$ has the \emph{bigon property}. Then we show that we can in fact isotope  $S$ to lie in a nicer position in which we say $S$ has the \emph{efficient bigon property}.
In \Cref{sec_negregion} we analyze surfaces with the efficient bigon property, show that we can promote this property to an even nicer one called the \emph{excellent bigon property} and show that up to isotopy all surfaces with the excellent bigon property are carried by $\tau^{(2)}$.
Finally, in \Cref{sec_export} we show how to use our techniques to prove a version of our results in the case when Dehn filling is not performed on any component of $\mr M$, which is used in \cite{LMT20}.

\subsection*{Acknowledgements.} We thank Samuel Taylor for feedback on earlier versions of this paper, and for many fruitful and interesting conversations. We also thank the anonymous referee, whose comments improved the clarity of the paper.

%%%%%%%%%%%%%%%%%%%%%%%%%%%%%%%%%%%%%%%%%%%%%%%%%%%%%%%%%%%%%%%%%%%%%%%%%%%%%%%%%%%%%%%%%%%%%%%%%%%%%%%%%%%%%%%%%%
\section{Background}\label{sec_bkgd}

\subsection{Euler characteristic, train tracks, index}\label{sec:index_bkgd}

Let $S$ be a compact surface, possibly with boundary, and let $t$ be a train track on $S$. By \textbf{train track}, we mean a 1-complex properly embedded in $S$ with a tangent space at each point varying continuously. We require that $t$ intersect the boundary of $S$ transversely. Following \cite{PenHar92} we call the points of $t\cap \partial S$ \textbf{stops}. The 1-cells of $t$ are called \textbf{branches} and the nodes of valence $>1$ are called \textbf{switches}. We make no assumptions about the types of complementary regions of $t$ or the valence of nodes in this notion of train track; however, we say a train track is \textbf{generic} if its switches are all trivalent. If a train track $t$ is generic, there is a vector field defined on the set of switches of $t$ defined up to scaling by the property that at each switch $p$ the vector field points away from the two-branched side and toward the one-branched side of $p$. To describe the direction of this vector field at a certain switch, we will say that the \emph{switch} points in that direction. If $b$ is a branch of $t$ such that neither endpoint of $b$ is a stop and the switches at both endpoints point into $b$, we say $b$ is a \textbf{large branch} of $t$.

A \textbf{patch of $(S, t)$}, or simply a \textbf{patch of $t$} if the surface in question is clear from context, is the closure of a component of $S-t$ with respect to a path metric. A patch of $t$ is topologically a surface with boundary, but a patch has additional data: the boundary of a patch may have cusps, which we call the \textbf{switches} of the patch, and there may also be \textbf{corners}. The switches and corners of a patch are the places where the patch's boundary is not $C^1$ and they correspond respectively to points where either $t$ has a switch or meets the boundary of $S$. (We use the word \emph{switch} instead of \emph{cusp} because we wish to reserve \emph{cusp} for a different usage later in the paper).

\begin{figure}
\centering
\includegraphics[height=2in]{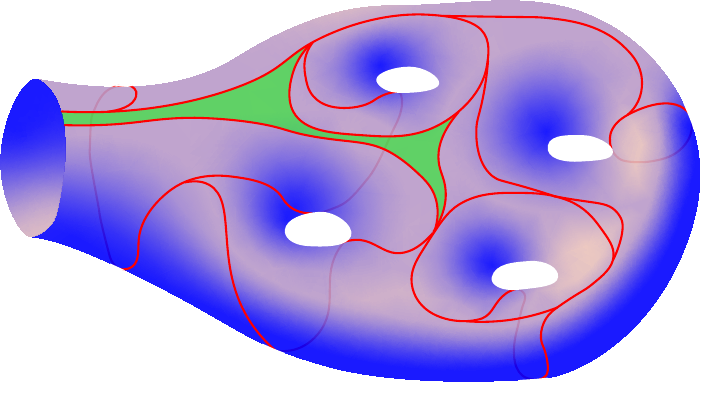}
\caption{A train track on a surface. The green patch is a topological disk with three switches and two corners, so its index is $2(1)-3-\frac{2}{2}=-2$.}
\label{patches}
\end{figure}

Define the \textbf{index} of a patch $p$ to be twice its Euler characteristic, minus its number of switches, minus half its number of corners:
\[
\ind(p)=2\chi(p)-\#\{\text{switches of $p$}\}-\frac{\#\{\text{corners of $p$}\}}{2}.
\]
For example, the index of a disk with 2 boundary cusps is $2(1)-2=0$. The Euler characteristic of $S$ can be obtained from the indices of its patches:
\begin{equation}\label{indexeq}
2\chi(S)=\sum \ind(p),
\end{equation}
where the sum is taken over the patches of $t$.

A patch which is topologically a disk and has $n$ switches and no corners will be called an \textbf{$\boldsymbol n$-gon patch} or simply an \textbf{$\boldsymbol n$-gon}. In particular we will refer to nullgons, monogons, and bigons in our arguments.

A patch which is topologically an annulus with no switches or corners in its boundary is called an \textbf{annulus}. To reduce confusion, if we wish to talk about a patch which is topologically an annulus but may have switches in its boundary we will refer to it as a \textbf{topological annulus}.

\subsection{Branched surfaces}

A \textbf{branched surface} is a 2-complex with a continuously varying tangent plane at every point, locally modeled on the quotient of a stack of disks by identifying closed half-disks (see \cite{Oer86} for more detail). Branched surfaces should be thought of as akin to train tracks: just as a train track embedded in a surface can carry curves and laminations with 1-dimensional leaves, a branched surface embedded in a 3-manifold can carry surfaces and laminations with 2-dimensional leaves. 

We now describe what it means for a branched surface to carry a surface. A regular neighborhood $N(B)$ of a branched surface $B$ has a foliation $\mc F$ whose leaves are  line segments intersecting $B$ transversely. If $\mc F$ is oriented, we say $B$ is a \textbf{cooriented branched surface}. If $S$ is an oriented surface embedded in the oriented 3-manifold $M$, there is a natural coorientation on $S$ determined by requiring that a coorientation vector at a point $p$ in $S$, when appended to a positive basis for the tangent space $T_pS$, gives a positive basis for $T_pM$. Suppose $B$ is a cooriented branched surface. If $S$ is embedded in $N(B)$ transverse to $\mc F$ so that the orientation of $\mc F$ agrees with the coorientation of $S$, we say $S$ is \textbf{carried} by $B$.

The non-manifold points of a branched surface $B$ form its \textbf{branch locus}, denoted $\brloc(B)$. The connected components of $B-\brloc(B)$ are called \textbf{sectors} of $B$.
A \textbf{branch curve} of a branched surface $B$ is the image of an immersion $S^1\looparrowright \brloc(B)$. A \textbf{branch segment} is the image of an immersion $[0,1]\looparrowright \brloc(B)$.

A \textbf{branched surface with generic branch locus} is a branched surface $B$ locally modeled on one of the two spaces shown in \Cref{fig_branchtypes}, which should be distinguished since our branched surfaces will always be embedded in orientable 3-manifolds. The non-manifold points of $\brloc(B)$ for a branched surface with generic branch locus are called \textbf{triple points}.

\begin{figure}
\begin{center}
\includegraphics[]{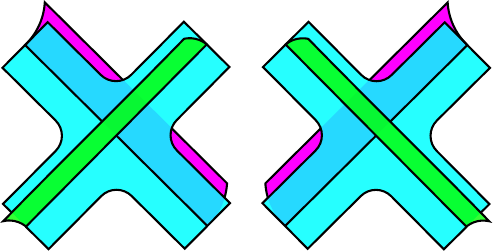}
\caption{The two local models for a branched surface with generic branch locus.}
\label{fig_branchtypes}
\end{center}
\end{figure}

Note that if $B$ is a branched surface in a 3-manifold $M$, and $S\subset M$ is a properly embedded surface in general position relative to $B$, then $B\cap S$ is a train track in $S$ and it makes sense to speak of the patches of $B\cap S$. If $\brloc(B)$ is generic, then $B\cap S$ is a generic train track.

We conclude this section by proving a basic lemma about incompressible surfaces in relation to branched surfaces in irreducible 3-manifolds. If $S$ is a surface embedded in $M$ in general position relative to a branched surface $B$, we say $S$ has \textbf{simple patches} if each patch of $S\cap B$ $\pi_1$-injects into its component of $M-B$. 

\begin{lemma}[simple patches]\label{lem_simplepatches}
Let $B$ be a branched surface in an irreducible 3-manifold $M$. Let $S$ be an incompressible surface embedded in $M$. Then $S$ can be isotoped so that $S\cap B$ has simple patches.
\end{lemma}

\begin{proof}
Suppose $P$ is a patch of $S\cap B$ in a component $T$ of $M-B$ and suppose that $P\to T$ is not $\pi_1$-injective. By the loop theorem there is some curve $\gamma\subset P$ which is essential in $P$ and bounds an embedded disk in $T$. If $\gamma$ does not bound a disk in $S$ then $D$ is a compression disk for $S$, a contradiction. Otherwise $\gamma$ bounds a disk $D'$ in $S$. Since $M$ is irreducible, $D\cup D'$ bounds a ball and by an innermost disk argument there is an isotopy of $S$ moving $D'$ into $T$. The effect on $S\cap B^s$ is to delete $D'\cap B^s$, strictly reducing the number of patches of $S\cap B^s$. After finitely many of these operations our surface will be of the desired form. 
\end{proof}

\subsection{Partial branched surfaces}\label{sec_pbs}

We will also need to consider objects called \emph{partial} branched surfaces, which are slightly more general than branched surfaces and to our knowledge have not been studied before. These are branched surfaces properly embedded in a submanifold of a 3-manifold.

\begin{definition}
Let $M$ be a compact 3-manifold. 
\begin{enumerate}[label=\alph*.]
\item A \textbf{partial branched surface} in $M$ is a 2-complex $B$ in $M$ such that there exists a union of solid tori $U\subset M$ such that $B$ is a properly embedded branched surface in $M-U$. 
\item An embedded surface $S\subset M$ is \textbf{carried by} $B$ if $S$ can be isotoped so that $S-U$ is carried by $B-U$ in $M-U$ and $S\cap U$ is a union of disks and annuli which $\pi_1$-inject into their respective components of $U$.
\end{enumerate}
\end{definition}

\subsection{The Thurston norm, its dual, and taut surfaces}\label{sec_tautsurfaces}

For a reference elaborating on the following see \cite{Thu86}. Let $M$ be a closed, oriented, irreducible, atoroidal 3-manifold. If $\alpha$ is an integral point in $H_2(M)$, define
\[
x(\alpha)=\min\{-\chi(S-\text{sphere components})\mid \text{$S$ embedded, oriented, and }[S]=\alpha\}.
\]
Thurston shows that $x$ extends to a norm on the vector space $H_2(M)$, now known as the \textbf{Thurston norm}, and that its unit ball $B_x(M)$ is a finite-sided rational polyhedron.

The \textbf{dual Thurston norm} is the norm $x^*$ on $H_1(M)$ defined by 
\[
x^*(\gamma)=\sup\{\langle \alpha, \gamma \rangle\mid x(\alpha) \le 1\},
\]
where $\langle \cdot, \cdot \rangle$ is algebraic intersection. We let $B_{x^*}(M)=\{\gamma\in H_1(M)\mid x^*(\gamma)=1\}$. Thurston shows that $B_{x^*}(M)$ is a polyhedron with integral vertices.

We say an embedded surface $S\subset M$ is \textbf{taut} if $-\chi(S)=x([S])$ and no component of $S$ is nulhomologous. Any taut surface is incompressible, since a compression surgery increases Euler characteristic.

More generally, Thurston shows that if $M$ admits surfaces of nonnegative Euler characteristic representing nonzero homology classes then $x$ gives a pseudonorm on $H_2(M)$. However, we will not need to consider this case.

\section{Veering triangulations: background and preliminaries}\label{sec_veering}

\subsection{Definition of \emph{veering triangulation}}\label{sec_veeringdef}
\begin{figure}
\centering
\includegraphics{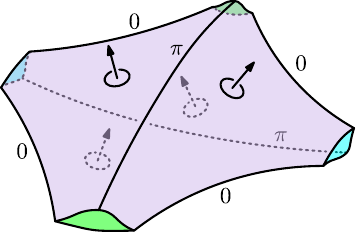}
\caption{A taut tetrahedron with edge angles and face coorientations indicated. The tips colored green are upward flat triangles, and the blue tips are downward flat triangles.}
\label{fig_tauttetrahedron}
\end{figure}

A \textbf{taut tetrahedron} is an ideal tetrahedron, i.e. a tetrahedron with its vertices removed, with the following extra structure: two faces are cooriented outward, two are cooriented inward, and the edges are labeled 0 and $\pi$ according to whether face coorientations disagree or agree along that edge, respectively. By pinching along the edges labeled 0 and smoothing along the edges labeled $\pi$, we always think of the $0,\pi$ labels as denoting the interior angle of the taut tetrahedron at that edge. This gives a taut tetrahedron a smooth structure in the sense that there is a well-defined tangent plane at every point in the boundary. See \Cref{fig_tauttetrahedron}.

A face of a taut tetrahedron $\taut$ is called a \textbf{top (bottom) face} if it is cooriented outward (inward). The union of the two top (bottom) faces of $\taut$ is called the \textbf{top (bottom)} of $\taut$. The edge along which the two top (bottom) faces of $\taut$ meet is called the \textbf{top (bottom) edge} of $\taut$.

A \textbf{taut ideal triangulation} is an ideal triangulation $\tau$ of a 3-manifold $N$ by taut ideal tetrahedra such that for each edge $e$, the sum of interior angles around $e$ is $2\pi$ (equivalently, $e$ is the image of exactly two edges labeled $\pi$ under the quotient map $\{\text{tetrahedra}\}\to N$) and such that bottom faces are only identified with top faces and vice versa. This forces $N$ to be orientable, and a choice of orientation for $N$ gives an orientation on the 2-skeleton $\tau^{(2)}$ by requiring that the direct sum of the orientation on each tangent plane with the coorientation be positively oriented. This gives a notion of clockwise and counterclockwise on each 2-cell of $\tau$.

Let $\tau$ be a taut ideal triangulation. By removing small neighborhoods of the ideal vertices, we will always think of $\tau$ as a decomposition of a compact manifold $\mr M$ into \emph{truncated} taut tetrahedra. A truncated taut tetrahedron $\tet$ has 8 sides, 4 of which are hexagons and 4 of which are triangles which we call the \textbf{tips} of $\tet$. Each tip has the smooth structure of a bigon. Instead of repeating ``truncated taut tetrahedron," we will use the term \textbf{$\boldsymbol\tau$-tetrahedron}. Our convention will be that the terms \textbf{$\boldsymbol\tau$-face} and \textbf{$\boldsymbol\tau$-edge} refer to ``honest" faces and edges of $\tau$, i.e. faces and edges intersecting $\intr(\mr M)$. If $\hex$ is a $\tau$-face, each edge of $\hex$ which is not a $\tau$-edge is called a \textbf{tip} of $\hex$. Each $\tau$-face has 3 tips.

\begin{figure}
\centering
\includegraphics[height=1.5in]{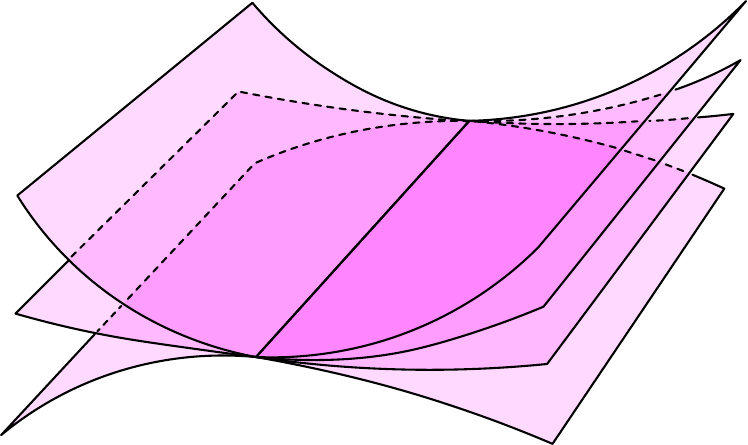}
\caption{The 2-skeleton of a taut ideal triangulation is a cooriented branched surface. Here we see a neighborhood of an edge of valence 7.}
\label{fig_itsabs}
\end{figure}

The 2-skeleton $\tau^{(2)}$ of $\tau$ can naturally be viewed as a cooriented branched surface. See \Cref{fig_itsabs}. The branch locus of $\tau^{(2)}$ is equal to the collection of $\tau$-edges.  Note that $\tau$ induces a cooriented train track 
\[
\partial\hbs:=\tau^{(2)}\cap \partial \mr M
\] 
on $\partial \mr M$. The patches of $\partial\hbs$ are all tips of $\tau$-tetrahedra, so for Euler characteristic reasons $\partial \mr M$ must be a union of tori. Let $U:=M-\intr(\mr M)$, so $U$ is a union of closed solid tori. 

A \textbf{flat triangle} $t$ is a bigon of a cooriented train track $T$ with three branches of $T$ in its boundary. An example is a tip of a $\tau$-tetrahedron with respect to $\partial\hbs$. The flat triangle $t$ has three vertices, two of which are switches of $t$; the internal angle is 0 at these vertices, which we call \textbf{cuspidal}. We will call the third vertex \textbf{noncuspidal}; the internal angle at this vertex is $\pi$. An edge of a flat triangle is a \textbf{0-0 edge} if it connects the two cuspidal vertices and otherwise is a \textbf{0-$\boldsymbol\pi$ edge}.
There are two types of flat triangles, distinguished by the coorientation at the noncuspidal vertex: a flat triangle $t$ is called \textbf{upward (downward)} if the coorientation at the noncuspidal vertex points out of (into) $t$. See \Cref{fig_updowntriangles}.

\begin{figure}
\centering
\includegraphics[width=3in]{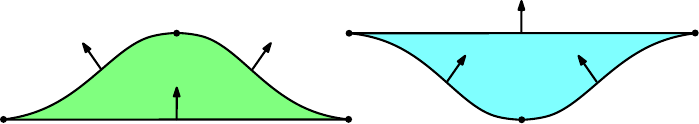}
\caption{An upward flat triangle (left) and a downward flat triangle (right).}
\label{fig_updowntriangles}
\end{figure}

An \textbf{upward (downward) ladder} is an annulus $A$ together with a cooriented train track $t$ containing $\partial A$ such that all patches of $t$ are upward (downward) flat triangles and for each branch $b$ of $t$, either $b\subset \partial A$ or $(b-\{\text{endpoints of $b$}\})\subset \intr A$ and both endpoints of $b$ lie on different components of $\partial A$. A \textbf{rung} is a branch of $t$ with endpoints on different components of $\partial A$.

\begin{figure}
\centering
\includegraphics[height=2in]{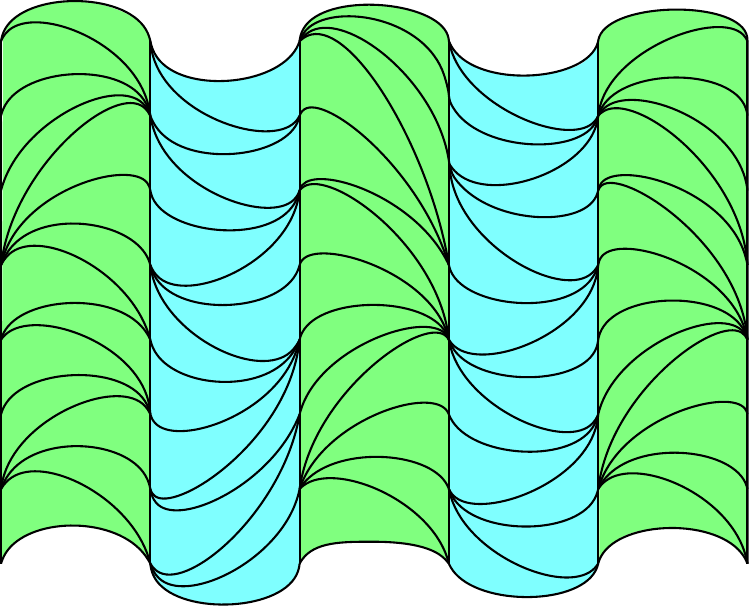}
\label{fig_flattriangulation}
\caption{Part of a pseudohyperbolic train track. We have colored the upward ladders green and the downward ladders cyan for compatibility with \Cref{fig_tauttetrahedron} and \Cref{fig_updowntriangles}.}
\end{figure}

A cooriented train track $t$ on a torus $T$ is \textbf{pseudohyperbolic} if $T$ is a union of ladders with disjoint interiors such that each upward (downward) ladder is disjoint from all other upward (downward) ladders.

\begin{definition}
A taut ideal triangulation $\tau$ of $\mr M$ is \textbf{veering} if it induces a pseudohyperbolic train track on each boundary component of $\mr M$.
\end{definition}

This definition is equivalent to Agol's original definition in \cite{Ago10} and the definition involving an edge bicoloring introduced by Hodgson-Rubinstein-Segerman-Tillmann in \cite{HRST11} which is now common. We choose this definition since the boundary train tracks of veering triangulations are integral to our arguments. To see that the definitions are equivalent, one notes that Agol's veering condition on edges corresponds to a certain condition on $\partial\hbs$ (see \cite[Lemma 2.6]{Lan18}), which implies the ladder behavior described here. 

%%%%%%%%%%%%%%%%%%%%%%%%%%%%%%%%%%%%%%%%%%%%%

\subsection{Veering combinatorics: ladderpoles, left/right veer, fans}\label{sec_veeringcombinatorics}

\begin{figure}
\centering
\includegraphics[height=2in]{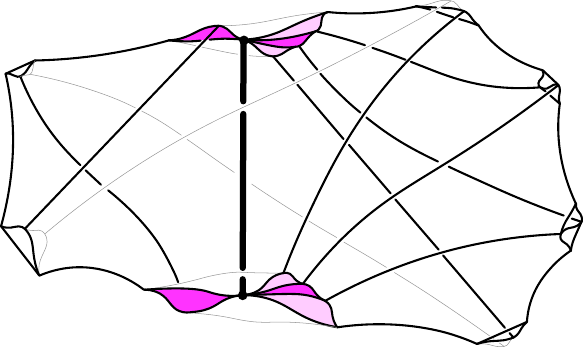}
\caption{In this picture we have drawn all the $\wt\tau$-tetrahedra incident to a bold $\wt \tau$-edge in the universal cover $\wt M^\circ$. The $\wt\tau$-tetrahedra which are not part of the two fans of the bold edge are outlined in gray. The bold $\wt\tau$-edge is incident to a long fan and a short fan. We have colored the flat triangles incident to a boundary point of the dark $\wt\tau$-edge to indicate their hinge- or non-hingeness: the light pink flat triangles are hinge and the darker pink triangles are non-hinge. The reader may wish to check that the bold edge is left veering.}
\label{fig_fanexample}
\end{figure}

Many of the terms defined in this subsection were coined in \cite{FutGue13}.

Let $T$ be a torus with a pseudohyperbolic train track. The boundary components of the ladders in $T$ are called \textbf{ladderpoles}. Together, the ladderpoles form an even-sized collection of parallel curves and determine a slope on $T$ called the \textbf{ladderpole slope}.

Let $\tau$ be a veering triangulation and $T$ be a component of $\partial \mr M$.  We define a notion of left and right on $T$ as follows. Let $L$ be a downward ladder; orient the core curve of $L$ so that it intersects each rung positively. If we look at $L$ from inside $\mr M$ so that our heads point in the direction of the core curve, the ladderpole of $L$ to our left (right) is called the left (right) ladderpole of $L$.

A \textbf{left (right) ladderpole} of $\tau$ is a ladderpole which is the left (right) ladderpole of a downward ladder. One can check that for every $\tau$-edge $e$, either both endpoints of $e$ lie in left ladderpoles or both lie in right ladderpoles. If both endpoints of $e$ lie in left (right) ladderpoles, then $e$ is 
\textbf{left (right) veering}. The quality of being left or right veering is called the \textbf{veer} of an edge. A $\tau$-tetrahedron is called \textbf{hinge} if its top and bottom edges have opposite veer, and \textbf{non-hinge} otherwise. We will also say that a tip of a $\tau$-tetrahedron $\tet$ is hinge or non-hinge if $\tet$ is hinge or non-hinge, respectively.

As with our definition of veering triangulation, this definition of veer agrees with others in the literature, and we have chosen to phrase it in this way in order to highlight what is most relevant to our arguments.

Let $\wt M^\circ$ be the universal cover of $\mr M$ and let $\wt \tau$ be the lift of $\tau$ to $\mr M$. We will refer to $\wt \tau$-tetrahedra, $\wt\tau$-faces, and $\wt \tau$-edges. It follows from \cite[Theorems 3.2 and 5.1]{SchSeg20} that no two $\wt\tau$-edges or $\wt\tau$-faces of a single $\wt \tau$-tetrahedron are identified, which makes it easier to define certain objects in $\mr M$ by first describing an object in $\wt M^\circ$ and then projecting to $\mr M$.

Let $e$ be a $\wt\tau$-edge. If we circle around $e$ and read off the interior angles at $e$ of $\tau$-tetrahedra incident to $e$, there are two $\pi$ angles and some number of $0$ angles. It is a consequence of our veering definition that the two $\pi$ angles are not circularly adjacent. Hence the $\wt\tau$-tetrahedra whose interior angles at $e$ are $0$ are split into two nonempty sets, each of which is called a \textbf{fan} of $e$. A fan is \textbf{short} if it consists of one $\wt\tau$-tetrahedron and \textbf{long} otherwise.

Let $v$ be a switch of $\partial \wt \tau$ corresponding to a $\wt\tau$-edge $e$. A \textbf{fan} of $v$ is the union of all upward, or all downward, flat triangles for which $v$ is a cuspidal vertex. This corresponds exactly to the intersection of one of the fans of $e$ with the component of $\partial \wt M^\circ$ containing $v$. A fan of $v$ is \textbf{short} if it consists of only one flat triangle and \textbf{long} otherwise. 
In $\mr M$, we define short and long fans of $\tau$-tetrahedra and flat triangles to be the images of the corresponding objects in $\wt M^\circ$ under the covering projection.

Note that the coorientation on $\partial\hbs$ allows us to speak of the \textbf{topmost} or \textbf{bottommost} flat triangle in a particular fan.
We now record some facts about fans and hingeness from \cite[Observations 2.6 and 2.7]{FutGue13} in the following lemma. 

\begin{lemma}[Futer--Gu\'eritaud]\label{lem_fans}
A non-hinge flat triangle is incident to a short fan at one of its cuspidal vertices. At the other cuspidal vertex, it is part of a long fan for which it is neither the top nor bottom flat triangle.

A hinge flat triangle is topmost in a fan corresponding to one of its cuspidal vertices, and bottommost in a fan corresponding to its other cuspidal vertex.
\end{lemma}

%%%%%%%%%%%%%%%%%%%%%%%%%%%%%%%%%%%%%%%%%%%%%%%%%%%%%%%%%%%%%%%%%%%%%%%%%%%%%%%%%%%%%%%%%%%%%%%%%%%%%%%%%%%%%%%%%%
\subsection{The stable and unstable branched surfaces}\label{sec:BSstruc}

\begin{figure}[h]
\centering
\includegraphics[width=5in]{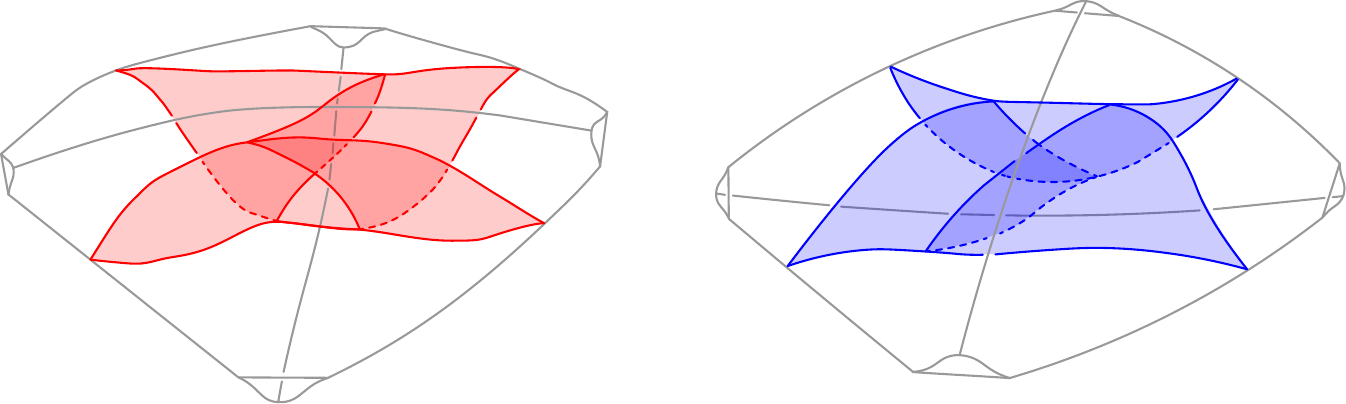}
\caption{Pieces of the stable branched surface $B^s$ (left) and the unstable branched surface $B^u$ (right). The mirror images of these pictures are also possible.}
\label{fig_stableunstable}
\end{figure}

Besides the 2-skeleton $\tau^{(2)}$, there are two branched surfaces naturally associated to the veering triangulation $\tau$. In this section we define them and discuss their combinatorics. As with our discussion of fans in \Cref{sec_veeringcombinatorics}, it is simplest to define these branched surfaces as the images of their lifts to the universal cover $\wt M^\circ$ of $\mr M$.

Let $\wt B^s$ be the branched surface in $\wt M^\circ$ defined as follows: $\wt B^s$ is topologically the 2-skeleton of the dual complex to $\wt \tau$. It is endowed with the smooth structure of a branched surface with generic branch locus determined by the following property: for any 2-cell $\hex$ of $\tau$, $\hex\cap B^s$ is a train track with one stop on each $\wt \tau$-edge of $\hex$, and a single switch which points to the bottom edge of the unique $\tau$-tetrahedron of which $\hex$ is a bottom $\tau$-face. We also require that $\wt B^s$ be preserved by $\pi_1(\mr M)$ acting by deck transformations on $\wt M^\circ$, which is possible since the deck group preserves our combinatorial notions of top and bottom. We define the \textbf{stable branched surface} of $\tau$, denoted $B^s$, to be the image of $\wt B^s$ in $\mr M$ under the covering projection.

The \textbf{unstable branched surface} of $\tau$, denoted $B^u$, is defined in the same style but with a different condition: we require $\wt B^u$ to have the cellular structure of the dual complex to $\wt \tau$ and to intersect each $\wt\tau$-face $\hex$ in a generic train track whose unique switch points toward the $\wt\tau$-edge which is the top $\wt\tau$-edge of the $\wt \tau$-tetrahedron of which $\hex$ is a top face. Then we let $B^u$ be the image of $\wt B^u$ under the covering projection.
\begin{remark}
In \cite{SchSeg19}, $B^s$ and $B^u$ are called the \emph{upper and lower branched surfaces in dual position} and are colored green and purple respectively. We have chosen to color them red and blue in analogy with the stable and unstable singular foliations of a pseudo-Anosov flow, which are usually colored red and blue.
\end{remark}

The \textbf{dual graph} of $\tau$, denoted $\Gamma$, is the 1-skeleton of the complex dual to $\tau$. There is a natural orientation on $\Gamma$, defined by the property that the orientation on each edge agrees with the coorientation of the corresponding 2-cell of $\tau$. Let $\wt \Gamma$ denote the lift of $\Gamma$ to $\wt M^\circ$.
Since $B^s$ and $B^u$ are two different smoothings of the same 2-complex, they are isotopic in $\mr M$. Their branch loci can both be naturally identified with $\Gamma$ such that triple points are identified with $\Gamma$-vertices. Thus we will always think of the branch curves and segments of $B^s$ and $B^u$ as being oriented compatibly with $\Gamma$.

We now discuss how $B^s$ and $B^u$ interact with the veers of $\tau$-edges. As before, it is convenient to work in $\wt M^\circ$.
Let $\hex$ be a $\wt\tau$-face. Then $\hex^s:=\hex\cap\wt B^s$ is a train track with a single generic switch, as is $\hex^u:= \hex \cap \wt B^u$. These train track switches recover the veer of $\wt \tau$-edges and behave in a controlled way with respect to $\partial \wt \tau^{(2)}$, as we describe in the next two lemmas. Each lemma is most easily understood via the included pictures, but we also provide formal statements.

\begin{lemma}[train tracks from veering data]\label{lem_traintrackfacts}
Let $\hex$ be a $\wt \tau$-face. Then $\hex^s=\hex\cap \wt B^s$ and $\hex^u=\hex\cap \wt B^u$ can be recovered from the veer of the $\wt \tau$-edges of $\hex$ as in the following picture,
\begin{center} \includegraphics[]{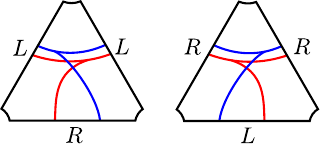} \end{center}
where the coorientation points out of the page and we have drawn $\hex^s$ in red and $\hex^u$ in blue. In words: 
\begin{itemize}
\item If $\hex$ has two \emph{left} veering $\wt\tau$-edges $e_1, e_2$ labeled so that $e_2$ follows $e_1$ in the counterclockwise order of the $\wt \tau$-edges of $\hex$ , then the switch of $\hex^s$ points toward $e_1$ and the switch of $\hex^u$ points toward $e_2$.
\item If $\hex$ has two \emph{right} veering $\wt \tau$-edges $e_1, e_2$ labeled so that $e_2$ follows $e_1$ in the counterclockwise order, then the switch of $\hex^u$ points toward $e_1$ and the switch of $\hex^s$ points toward $e_2$.
\end{itemize}
\end{lemma}

\begin{proof}
The reader can check this in order to get comfortable with the combinatorics of veering triangulations.
\end{proof}

\begin{remark}
We emphasize that while we have specified the position of $B^s$ and $B^u$ relative to $\tau^{(2)}$, we have \emph{not} specified the position of $B^s$ and $B^u$ relative to each other, and we will not do so. The salient feature of the train tracks in the lemma statement is the configuration of each train track relative to the right and left veering edges, not relative to the other train track.  
\end{remark}

Now suppose that $\hex$ is a $\wt\tau$-face with tips labeled $t_1,t_2, t_3$. Each $t_i$ determines a unique patch $t_i^s$ of $\hex\cap \wt B^s$ and $t_i^u$ of $\hex\cap \wt B^u$. Note that for each $i$, at most one of $t_i^s, t_i^u$ has a switch.

\begin{lemma}\label{lem_tipcondition}
With notation as above, $t_i$ is a ladderpole branch of $\partial \wt \tau^{(2)}$ if and only if $t_i^s$ and $t_i^u$ have no switches, and a rung of $\partial\wt \tau^{(2)}$ crossing an upward (downward) ladder if and only if the patch $t_i^s$ ($t_i^u$) has a switch. (See \Cref{fig_tipcondition}.)
\end{lemma}

\begin{figure}[!h]
\centering
\includegraphics[]{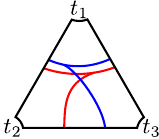}
\caption{With tips labeled as in the picture, $t_1$ is a ladderpole branch, $t_2$ is a rung crossing an upward ladder, and $t_3$ is a rung crossing a downward ladder.}
\label{fig_tipcondition}
\end{figure}

\begin{proof}
We recall that both ends of a left (right) veering edge meet $\partial \mr M$ in the left (right) ladderpole of a downward ladder. The lemma is now a consequence of \Cref{lem_traintrackfacts}. 
\end{proof}

\subsection{Branch curves, stable loops, unstable loops}

Following \cite{SchSeg20} we define a \textbf{normal curve} for $\tau$ to be an oriented loop $\gamma$ which is smoothly immersed in $\tau^{(2)}$, transverse to $\tau$-edges, and does not ``backtrack" in the sense that if $\wt \gamma$ is a lift of $\gamma$ to $\wt M^\circ$ then for each $\wt \tau$-face $\hex$ intersected by $\wt \gamma$, each component of $\wt\gamma\cap \hex$ has its endpoints on different $\wt\tau$-edges.
% (in fact, since Schleimer and Segerman prove normal curves are homotopically nontrivial, there will be only one component of $\wt\gamma^{-1}(\hex)$).

Note that for any $\wt \tau$-edge $e$ and for each fan of $e$, the coorientation on $\tau^{(2)}$ totally orders the $\wt\tau$-faces incident to $e$ belonging to $\wt \tau$-tetrahedra in the fan.
A normal curve $\gamma$ is a \textbf{stable loop} if a lift $\wt \gamma$ to $\wt M^\circ$ has the following property: at each $\wt \tau$-edge $e$ intersected by $\gamma$, $\gamma$ passes from a non-topmost face incident to $e$ to a topmost one (recall $\gamma$ is oriented). Similarly, $\gamma$ is an \textbf{unstable loop} if a lift $\wt \gamma$ to $\wt M^\circ$ has the property that at each $\wt \tau$-edge $e$ intersected by $\gamma$, $\gamma$ passes from a bottommost face incident to $e$ to a non-topmost one. These curves were first studied in \cite{Lan19}.

%(\textbf{unstable loop})
%(bottommost)
%(non-bottommost) 

Let $\gamma$ be a stable loop and $\wt \gamma$ its lift to $\wt M^\circ$. We say $\gamma$ is \textbf{shallow} if, whenever $\wt\gamma$ traverses a $\wt \tau$ edge, it passes from a second-topmost $\wt\tau$-face to a topmost $\wt\tau$-face.

Similarly, let $\gamma$ be an \emph{unstable} loop and $\wt \gamma$ its lift to $\wt M^\circ$. We say $\gamma$ is \textbf{shallow} if, whenever $\wt\gamma$ traverses a $\wt \tau$-edge, it passes from a bottommost $\wt\tau$-face to a second-bottommost $\wt\tau$-face.

Let $\gamma$ be a stable loop. Then by perturbing $\gamma$ slightly in the direction of the coorientation of $\tau^{(2)}$, we can homotope $\gamma$ to a curve which is positively transverse to $\tau^{(2)}$. Similarly if $\eta$ is an unstable loop, we can homotope $\eta$ slightly \emph{against} the coorientation to produce a closed positive transversal to $\tau^{(2)}$. The images of $\gamma$ and $\eta$ under these homotopies are called the \textbf{pushup} and \textbf{pushdown} of $\gamma$  and $\eta$, respectively. See \Cref{fig_pushupdown}.

If a normal curve is both a stable and unstable loop, meaning it always passes from bottommost to topmost faces,
we call it a \textbf{normal branch loop}. This name is justified by the following lemma.

\begin{figure}
\centering
\includegraphics[height=1 in]{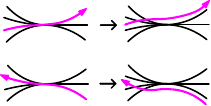}
\caption{Side views of pushing up a stable loop (top) and pushing down an unstable loop (bottom), with orientation indicated.}
\label{fig_pushupdown}
\end{figure}

\begin{lemma}[Characterization of branch curves]\label{lem_branchlines}
Suppose $\gamma$ is a normal branch loop. Let $\gamma^+$ and $\gamma^-$ be the pushup and pushdown of $\gamma$, respectively. Then the $\Gamma$-cycles determined by $\gamma^+$ and $\gamma^-$ are branch curves of $B^u$ and $B^s$, respectively. Furthermore, every branch curve of $B^s$ and $B^u$ can be obtained in this way.
\end{lemma}

\begin{proof}
Recall that we can identify the branch locus of $B^s$ with the dual graph $\Gamma$. We will first prove that the $\Gamma$-cycle determined by $\gamma^-$, which we will also call $\gamma^-$, is a stable branch curve (identifying $\Gamma$ with the branch locus of $B^s$). Let $\wt\gamma^-$ be a lift of $\gamma^-$ to the universal cover $\wt M$, and write $\wt\gamma^-$ as a sequence of $\wt \Gamma$-edges $(\dots, e_{-1},e_0, e_1,e_2\dots)$. Each ordered pair $(e_i,e_{i+1})$, in the language of \cite{LMT20}, forms a \emph{turn} of $\wt \gamma^-$. A turn $(e_i,e_{i+1})$ is \emph{branching} if $e_i\cup e_{i+1}$ is a branch segment of $\wt B^s$. We must show that every turn of $\wt \gamma^-$ is branching.

Let $f_i$ be the $\wt\tau$-face dual to $e_i$. Note that for all $i$, $f_i$ and $f_{i+1}$ lie respectively on the bottom and top of a single $\wt\tau$-tetrahedron $\tet_i$ and meet along a single $\wt\tau$-edge $f_i\cap f_{i+1}$. We have the following facts:
\begin{enumerate}[label=(\alph*)]
\item $(e_i, e_{i+1})$ is branching if and only if $f_i\cap f_{i+1}$ veers oppositely from the top edge of $\tet_i$ (\cite[Lemma 4.5]{LMT20}), and 
\item for any $\wt\tau$-edge $a$, the top $\wt\tau$-edge of a tetrahedron $\tet$ for which $a$ is a $0$-edge has the same veer as $a$ if and only if $\tet$ is topmost in its fan at $a$ (\cite[Fact 1]{LMT20}).
\end{enumerate}

Let $v_i$ be the terminal $\wt\Gamma$-vertex of $e_i$ and the initial vertex of $e_{i+1}$, so that $\tet_i$ is the unique $\wt\tau$-tetrahedron containing $v_i$. If $v_i$ is non-topmost in the fan of a $\wt\tau$-edge $a$ traversed by $\gamma$, then the top edge of $\tet_i$ veers oppositely from $a$ by fact (b), so fact (a) implies that $(e_i,e_{i+1})$ is a branching turn. 

Otherwise $v_i$ is topmost in the fan of a $\wt\tau$-edge $a$ traversed by $\wt\gamma$. Because $\gamma$ is a normal branch loop, each $\wt\tau$-edge traversed by $\gamma$ is the top $\wt\tau$-edge of the topmost $\wt\tau$-tetrahedron in a fan of the previous $\wt\tau$-edge traversed by $\wt\gamma$, so all $\wt\tau$-edges traversed by $\wt \gamma$ have the same veer by fact (b). If $\hex$ is the $\wt\tau$-face atop $\tet_i$ traversed by $\wt\gamma$, then this fact gives that $\hex\cap f_i=a$ has the same veer as the top $\wt\tau$-edge of $\tet_i$. Since $f_{i+1}$ is the other face besides $\hex$ atop $\tet_i$, the veer of $f_i\cap f_{i+1}$ is opposite from that of the top $\wt\tau$-edge of $\tet_i$. Hence $(e_i,e_{i+1})$ is a branching turn by fact (a).
The proof that the pushup of a normal branch loop is an unstable branch curve is symmetric.

It remains only to prove the statement that each branch curve of $B^s$ is the pushdown of a normal branch loop and that each branch curve of $B^u$ is the pushup of a normal branch loop. Let $c$ be a branch curve of $B^s$, and let $\wt c$ be its lift to $\wt M$. Then \Cref{lem_tipcondition} implies that there is an upward ladder $L$ such that $\wt c$ passes through exactly the $\tau$-tetrahedra that have a tip meeting the interior of $L$. From the picture of an upward ladder it is clear that $c$ can be isotoped into the $\tau$-faces corresponding to either ladderpole of $L$. Both of these isotopies are the inverses of pushdowns. Again, the proof for branch curves of $B^u$ is symmetric.
\end{proof}

Let $\gamma$ be a branch line of $B^s$ or $B^u$. Then as in the proof of the previous lemma, by \Cref{lem_tipcondition} there is an upward or downward ladder such that $\gamma$ passes through exactly the $\tau$-tetrahedra that possess a tip meeting the interior of $L$. It follows immediately from \Cref{lem_fans} that $L$ contains hinge flat triangles: if $t$ is a flat triangle in $L$ which is non-hinge, then by \Cref{lem_fans} $t$ is part of a long fan on one of its sides, the top and bottom flat triangles of which must be hinge. This implies the following lemma.

\begin{lemma}\label{lem_branchlinehinge}
Every branch line of $B^s$ or $B^u$ passes through a hinge $\tau$-tetrahedron.
\end{lemma}

We conclude our development of veering combinatorics by giving a characterization of shallow stable/unstable loops. 

Let $L\subset \del\mr M$ be a ladder, and let $\wt L$ be a lift of $L$ to $\wt M^\circ$. Let $S(\wt L)$ be the union of all $\wt\tau$-tetrahedra that meet $\wt L$ in its interior (for mnemonic purposes, $S(\wt L)$ is a cellular \emph{saturation} of $\wt L$ in $\wt M^\circ$). Let $(\tet_n)_{n\in \Z}$ be the bi-infinite sequence of all $\wt\tau$-tetrahedra contained in $L$, ordered so that for all $i$, a top face of $\tet_{i}$ is identified with a  bottom face of $\tet_{i+1}$.
Assume for the moment that $L$ is an upward ladder, and
let $e_i$ be the bottom $\wt\tau$-edge of $\tet_i$.
Note that each $\tet_i$ contains a unique face, call it $\hex_i$, which does not meet $\wt L$. Since $\wt L$ is upward, $\hex_i$ lies in the bottom of $\tet_i$; we let $\hex_i'$ be the other $\wt\tau$-face in the bottom of $\tet_i$. Observe that $\hex_i'$ is the unique face along which $\tet_{i-1}$ and $\tet_i$ meet. This implies that $\hex_i$ is topmost in a fan of $e_i$ and second-topmost in a fan of $e_{i+1}$. Hence the core of the bi-infinite strip $\bigcup_i \hex_i$, which we can choose to be invariant under the action of the deck transformation corresponding to the core of $L$ and preserving $\wt L$, projects to a shallow stable loop in $\mr M$.
A symmetric discussion holds if $L$ is downward, where we produce a shallow \emph{unstable} loop by considering the $\wt\tau$-faces in $S(\wt L)$ not meeting $\wt L$.

We now show that this is the only way shallow stable/unstable loops arise. In the following argument we use notation to mirror the above construction of a shallow stable loop. Suppose that $\gamma$ is a shallow stable loop in $M$ and let $\wt\gamma$ be a lift to $\wt M^\circ$. Let $(\hex_n)_{n\in \Z}$ be the sequence of $\wt\tau$-faces traversed by $\wt\gamma$, let $\tet_i$ be the $\wt\tau$-tetrahedron atop $\hex_i$, let $\hex_i'$ be the bottom $\wt\tau$-face of $\tet_i$ distinct from $\hex_i$, and let $e_i$ be the bottom $\wt\tau$-edge of $\tet_i$. Note that $\hex_i$ and $\hex_{i+1}$ meet along $e_{i+1}$. Let $t_i$ be the tip of $\tet_i$ disjoint from $\hex_i$ and note that $t_i$ is an upward flat triangle. Since $\gamma$ is a shallow stable loop, $\hex_i$ is second-topmost in a fan of $e$ and $\hex_{i+1}$ is topmost in the other fan of $e$. It follows that $\tet_{i+1}$ meets $\tet_i$ along a bottom face, which is necessarily $\hex_{i+1}'$. This implies that the tip $t_{i+1}$ meets $t_i$ along a branch of $\partial \wt\tau^{(2)}$. Therefore $\bigcup_i t_i$ is the lift of an upward ladder in $\partial \mr M$. As in the construction above, there is a parallel argument showing that every shallow unstable loop gives rise to a downward ladder. We record what we have learned in the following lemma.

\begin{lemma}[Characterization of shallow stable and unstable loops]\label{lem_shallowloops}
If $\wt L$ is the lift to $\wt M^\circ$ of an upward (downward) ladder $L$, and $S(\wt L)$ is defined as above, then the union of all faces in $S(\wt L)$ not meeting $\wt L$ carries a curve covering a shallow stable (unstable) loop in $M$. Furthermore, all shallow stable (unstable) loops can be obtained in this way.
\end{lemma}

%%%%%%%%%%%%%%%%%%%%%%%%%%%%%%%%
%%%%%%%%%%%%%%%%%%%%%%%%%%%%%%%%
%%%%%%%%%%%%%%%%%%%%%%%%%%%%%%%%
%%%%%%%%%%%%%%%%%%%%%%%%%%%%%%%%

\section{Flattened surfaces}\label{sec_flatsurf}

\subsection{Conventions and preliminaries to flattening}\label{sec_flatprelim}

We now embark on the proof of Theorems \ref{thma} and \ref{thmb}. We first need to define the notion of prongs from the theorem statement. Let $t$ be a pseudohyperbolic train track on a torus $T$, let $\lambda\subset T$ be the union of all ladderpoles of $t$, and let $s$ be a slope on $T$. Then we define 
\[
\prongs(s):=\frac{i_g(s,\lambda)}{2}
\]
where $i_g$ is geometric intersection number. We say $\prongs(s)$ is the number of prongs of the slope $s$.

For the remainder of the paper, except the final section, we will be in the following situation:
\begin{itemize}
\item $M$ is a closed 3-manifold,
\item $U$ is a union of solid tori in $M$,
\item $\tau$ is a veering triangulation of $\mr M:= M-\intr(U)$, and 
\item the meridional slope for each component of $U$ has $\ge3$ prongs with respect to $\partial\hbs$.
\end{itemize}

Hence we will think of $\mr M$ as a closed submanifold of $M$, and $\tau^{(2)}$ as a partial branched surface in $M$.

\begin{lemma}\label{lem_irreducible}
$M$ is irreducible.
\end{lemma}
\begin{proof}
As is pointed out in \cite[\S 6.5]{SchSeg19}, $B^s$ and $B^u$ are {laminar branched surfaces} in $M$ in the sense of Li. Hence \cite[Theorem 1]{Li02} gives that $M$ contains an {essential lamination} and hence is irreducible by \cite[Theorem 1]{GO89}.
\end{proof}

In this section we discuss a process called \emph{flattening} which takes an embedded surface and isotopes it so that $S\cap \mr M$ lies in a regular neighborhood of $\tau^{(2)}$ in a controlled fashion. This will be a useful way to keep track of surfaces as we move them around in $M$.

\subsection{The rod and plate neighborhood of $\tau^{(2)}$}\label{sec_rodplate}
\begin{figure}
\centering
\includegraphics[width=4in]{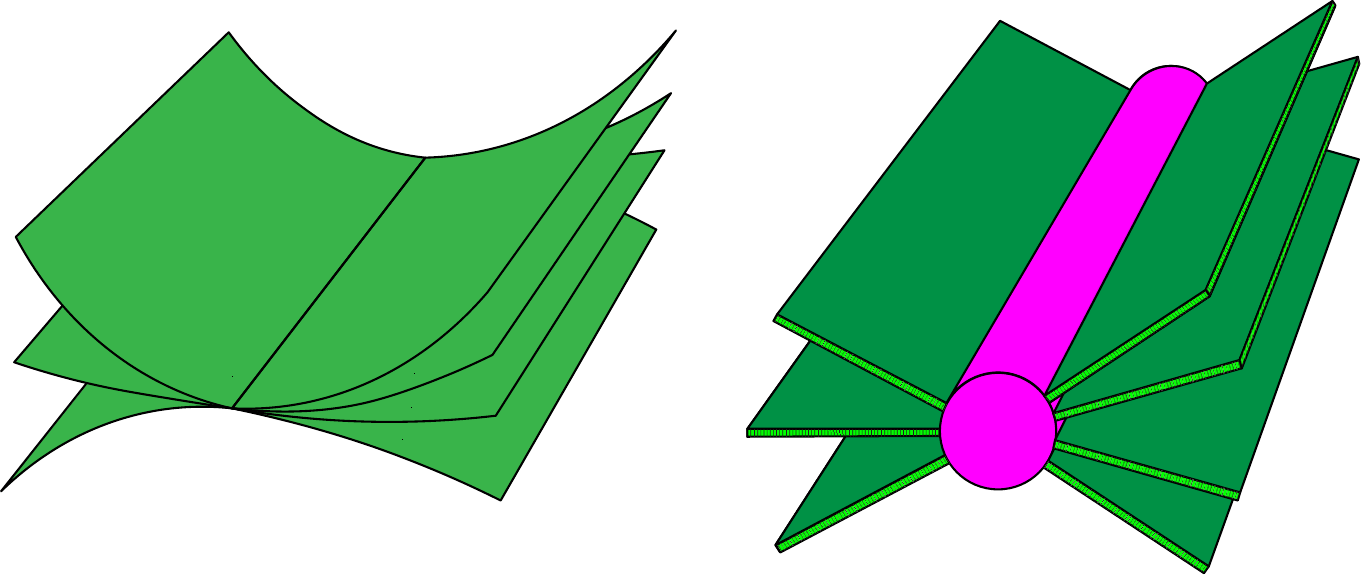}
\caption{Left: local picture around a $\tau$-edge. Right: the corresponding portion of $N_\epsilon$, decomposed into part of a rod (magenta) meeting various plates (green). 
}
\label{fig_nbhd}
\end{figure}

We first construct a regular neighborhood of $\tau^{(2)}$. For each $\tau$-edge $e$, let $R_e:= e\times D^2$ be the \textbf{rod corresponding to $e$} (where $D^2$ is a 2-dimensional disk). For a $\tau$-face $\hex$, let $P_\smhex:=\hex\times [0,1]$ be the \textbf{plate corresponding to} $\hex$. Then $P_{\smhex}$ has a natural foliation $\bigcup_{p\in \smhex}\{p\}\times [0,1]$ by line segments that we call the \textbf{vertical foliation} of $P_\smhex$. We orient the leaves of the vertical foliation so that, identifying $\hex$ with $\hex\times\{0\}$, the coorientation of $\hex\times\{0\}$ agrees with the orientation of the leaves.

We can glue the rods and plates along their boundaries in a way prescribed by the identifications of $\tau$-edges and boundaries of $\tau$-faces to form the \textbf{rod and plate neighborhood of $\tau$}, denoted $N_\epsilon$. This gives a local picture as in the righthand side of \Cref{fig_nbhd}. We view $N_\epsilon$ as embedded in $\mr M\subset M$, with $\tau^{(2)}$ lying in its interior so that $\tau^{(2)}$ intersects each leaf of the vertical foliation in each plate in exactly one point such that the vertical leaf orientations are compatible with the coorientation of $\tau^{(2)}$.  We call the union of the vertical foliations in each plate the \textbf{vertical foliation of $N_\epsilon$}. Note that the rods are not foliated by the vertical foliation. There is a homotopy equivalence 
\[
\coll \colon N_\epsilon\twoheadrightarrow \tau^{(2)}
\] 
called the \textbf{collapsing map} which is given by collapsing the disk factors of each rod and collapsing the vertical leaves in plates. This map restricted to $\tau^{(2)}$ is nearly the identity.

Let $\partial N_\epsilon=N_\epsilon\cap \partial \mr M$.  
There is a natural decomposition of $\partial N_\epsilon$ into \textbf{junctions} and \textbf{conduits}, which are respectively the components of intersections of rods and plates with $\partial \mr M$. The conduits inherit an oriented foliation from the vertical foliation, and the union of these foliations over all conduits is called the \textbf{vertical foliation of $\partial N_\epsilon$}. Similarly to $N_\epsilon$, we are leaving the junctions unfoliated.

Before beginning the flattening process, it is convenient to assume that each leaf of the vertical foliation that intersects $B^s$ is tangent to and contained in $B^s$. This can be achieved by a small isotopy.

\subsection{The flattening process}\label{sec_flatteningprocess}
Let $\Sigma$ be a sector of $\wt B^s$. Because the underlying topology of $\wt B^s$ is that of the dual 2-complex to $\wt\tau$, $\Sigma$ is a topological disk pierced by a single $\wt\tau$-edge $e$. Further, the intersection of $\wt\tau^{(2)}$ with $\closure(\Sigma)$ is a train track with a single switch at $e\cap \Sigma$ and one stop in each component of $\partial\Sigma-\{\text{triple points of $\wt B^s$}\}$. See \Cref{fig_sectorintersection}.

\begin{figure}[h]
\centering
\includegraphics[height=1.5in]{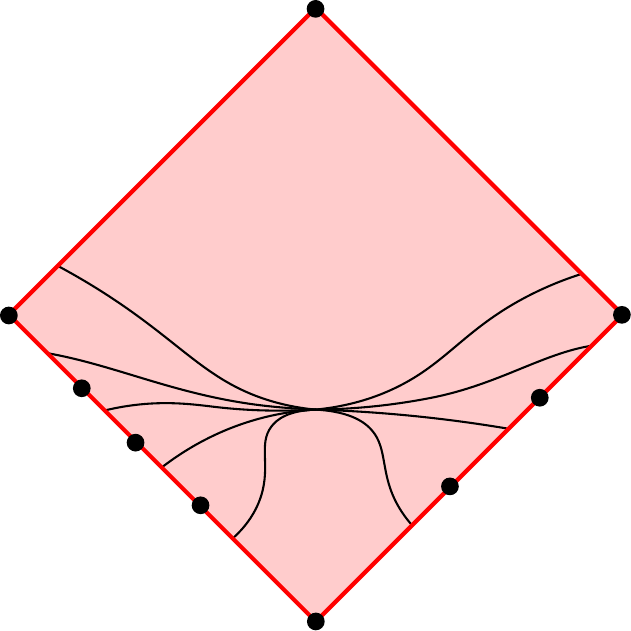}
 \caption{The intersection of $\tau^{(2)}$ with a single sector of $B^s$. While the straight lines in the boundary of the sector denote branch segments of $B^s$, we do not use this additional information.}
\label{fig_sectorintersection}
\end{figure}

This duality of complexes also implies that the complementary components of $B^s$ in $M$ are homeomorphic to open solid tori. Let $S$ be a taut surface. By \Cref{lem_simplepatches}, we can assume that $S\cap B^s$ has simple patches. Hence each patch is either a meridional disk, a nonmeridional disk, or a $\pi_1$-injective annulus with respect to the component of $M-B^s$ in which it resides.

\medskip

\textbf{Step 1.} We first perform an isotopy supported in a regular neighborhood of $B^s$ so that $S\cap B^s$ lies in $N_\epsilon\cap B^s$ transverse to the induced vertical foliation. This can be done as follows. Let $\Sigma$ be a sector of $B^s$, containing a single $B^s\cap \tau^{(2)}$-switch $v_\Sigma$. 
Since $S\cap B^s$ has simple patches, $S\cap\Sigma$ contains no circle components and is thus a collection of arcs with endpoints in $\partial \Sigma$. We can isotope  $S$ so that each of these arcs lies in $N_\epsilon\cap \Sigma$ transverse to the vertical foliation. This is shown in \Cref{fig_flattening12}. This isotopy can be performed consistently for all sectors of $B^s$.

\begin{figure}[h]
\centering
\includegraphics[width=4.5in]{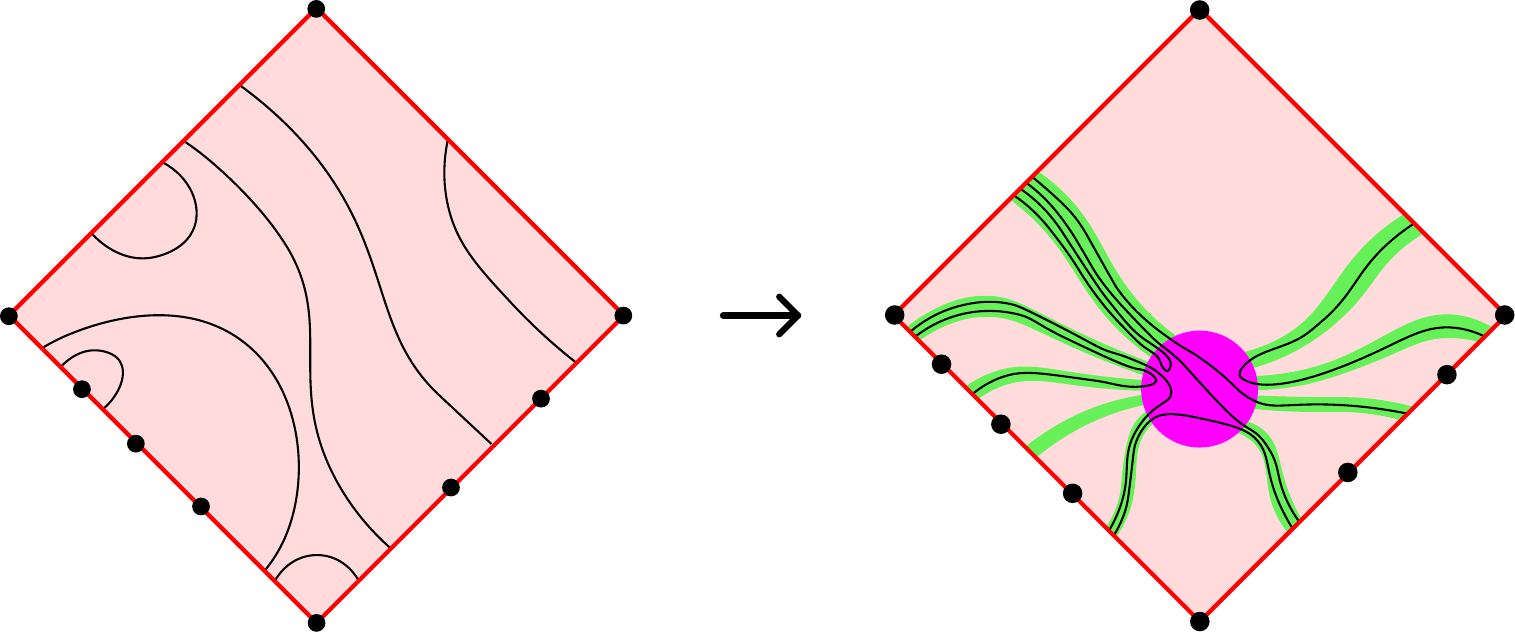}
 \caption{Step 1 of the flattening process in a single $B^s$-sector. We have not drawn the vertical foliation.}
\label{fig_flattening12}
\end{figure}

\textbf{Step 2.} After step 1, $S$ lies close to $\tau^{(2)}$ near $B^s$. Consider a complementary region $R$ of $B^s$; as mentioned before step 1, $R$ is homeomorphic to an open solid torus, and contains a single component $U_R$ of $U$. Let $T=\partial U_R$. Again by duality, we can identify $\tau^{(2)}\cap R$ with $(\partial\hbs\cap T)\times [0,1)$. Using this product structure, we can isotope  $S$ so that $S\cap R\cap \mr M$ lies in $N_\epsilon\cap T$ transverse to the vertical foliation. This can be visualized as ``combing" $S$ onto $\hbs$ and into $U_R$. After performing this operation in all complementary regions of $B^s$, the process is complete.

\begin{remark}
We have broken symmetry slightly by flattening with respect to $B^s$. We could also flatten with respect to $B^u$; the essential thing about $B^s$ that we use is its duality with $\tau$. 
\end{remark}

\subsection{After flattening}
Now we reckon with the aftermath of the flattening process. We call the resulting surface the \textbf{flattening} of $S$ and use the notation $S^\flat$. Any surface $F$ such that each component of $F\cap U$ $\pi_1$-injects into $U$ and $F\cap \mr M$ is embedded in $N_\epsilon$ transverse to the vertical foliation is called a \textbf{flattened surface}.  A connected component of $S^\flat\cap \{\text{plates of $N_\epsilon$}\}$ is called a \textbf{plate} of $S^\flat$. Similarly a connected component of $S^\flat\cap \{\text{rods of $N_\epsilon$}\}$ is called a \textbf{rod} of $S^\flat$. The total number of plates in $S^\flat$ is called the \textbf{area} of $S^\flat$.
Let 
$\mr S^\flat=S^\flat\cap \mr M$.

In each plate of $S^\flat$, the pairing of the coorientation of $S^\flat$ with the orientation of the vertical foliation is either entirely positive or negative, and in these respective cases we call the plate a \textbf{positive plate} or a \textbf{negative plate}. Each rod $r$ of $S^\flat$ is incident to two plates, and if both plates are positive (negative) we say $r$ is a \textbf{positive (negative) rod}. If these two plates have opposite sign we say $r$ is a \textbf{mixed rod}.

Let $\Neg(S^\flat)$ be the subsurface of $S^\flat$ obtained by taking the union of all negative plates and negative rods. Thus $S^\flat$ is carried by the partial branched surface $\tau^{(2)}$ exactly when $\Neg(S^\flat)=\varnothing$. A detailed analysis of $\Neg(S^\flat)$ for certain flattened surfaces will be central to our proof of \Cref{thm_mainthmveering}.

We now describe the effects of flattening in relation to $\partial\hbs$. 
\begin{itemize}

\item If $p$ was a meridional disk patch of $S$, then its image under flattening intersects $U$ in a meridional disk whose boundary lies in $\partial N_\epsilon$ transverse to the vertical foliation.

\item Similarly if $p$ was an annulus patch of $S$, then its image under flattening intersects $U$ in an annulus whose boundary lies in $\partial N_\epsilon$ transverse to the vertical foliation, and whose core is homotopically nontrivial in $U$.

\item Finally, if $p$ was a nonmeridional disk patch, then its image under flattening intersects $U$ in a nonmeridional disk. The boundary of this nonmeridional disk is a component of $\partial \mr S^\flat$ lying in $\partial N_\epsilon$ transverse to the vertical foliation, bounding a disk
$\delta$ in $\partial U$.
We will call $\delta$ a \textbf{$\boldsymbol{\flat}$-disk} of $S^\flat$. A $\flat$-disk $\delta$ is \textbf{inward} (\textbf{outward}) if the coorientation of $S^\flat$ points into (out of) $\delta$ along $\partial \delta$.
A $\flat$-disk is \textbf{innermost} if it contains no other $\flat$-disks in its interior. 
\end{itemize}

Let $\delta$ be a $\flat$-disk of $S^\flat$.
The \textbf{volume} of $\delta$ is the number of components of $\partial \mr M-\partial N_\epsilon$ contained in $\delta$. 
The \textbf{circumference} of $\delta$ is the length of the image of $\delta$ under the collapsing map $\coll|_{\partial\mr M}\colon\partial N_\epsilon\to \partial\hbs$, with respect to a metric assigning length 1 to each $\partial\hbs$-branch. 
Let $T_\delta$ be the component of $\partial\mr M$ containing $\delta$. Let $\wt\delta$ be a lift of $\delta$ to the universal cover $\wt T_\delta$ of $T_\delta$. The \textbf{width} of $\delta$ is the number of ladders $\wt L$ in $\wt T_\delta$ with the property that the image of $\partial\wt\delta$ under collapsing intersects both ladderpoles of $\wt L$.

\begin{figure}
\centering
\includegraphics[height=1.75in]{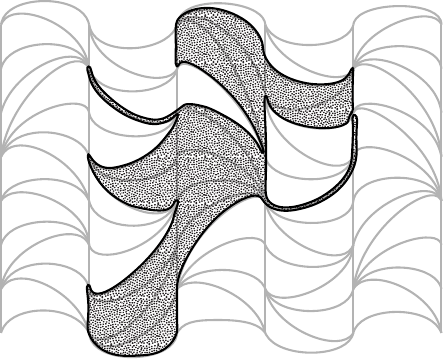}
\caption{An example of a $\flat$-disk with volume 16, width 3, and circumference 22. }
\label{fig_deldisk}
\end{figure}

The boundary of $\delta$ decomposes into cooriented line segments contained alternately in junctions and conduits of $\partial N_\epsilon$, which we call \textbf{junction segments} and \textbf{conduit segments} of $\partial\delta$, respectively. If there is no chance of confusion with junctions or conduits of $\partial N_\epsilon$ we will sometimes refer to a junction segment or conduit segment of $\partial\delta$ as simply a \textbf{junction} or \textbf{conduit} of $\delta$. A conduit segment is \textbf{positive} or \textbf{negative} if its coorientation respectively agrees or disagrees with the orientation of the vertical foliation. It is called a \textbf{ladderpole} or \textbf{rung conduit segment} if it intersects a ladderpole or rung of $\partial\hbs$, respectively. A junction segment is \textbf{positive (negative)} if connects two positive (negative) conduit segments, and \textbf{mixed} otherwise.

Usually when drawing a $\flat$-disk $\delta$ we omit drawing the junctions and conduits of $\partial N_\epsilon$, simply drawing $\partial \delta$ close to $\partial\hbs$ in such a way that it is clear which junctions and conduits $\partial \delta$ traverses. See \Cref{fig_deldisk}.

A \textbf{cusp} of $\delta$ is a 
mixed junction segment which is a ``convex" part of $\partial \delta$. See \Cref{fig_cusp}. To make this notion of convexity precise, let $c_1$ and $c_2$ be conduit segments of opposite sign connected by a junction segment $j$ contained in a junction $J$. Then near $J$, one of the $c_i$ is higher with respect to the coorientation of $\partial\hbs$; suppose that $c_1$ is higher. If the orientation of the vertical foliation points out of $\delta$ along $c_1$, then $j$ is a cusp. Let $B_J$ be a small open neighborhood of $J$. Let $B_J(j)$ be the component of $\delta\cap B_J$ whose boundary in $B_J$ is $B_J\cap(c_1\cup c_2\cup j)$. The \textbf{size} of $j$ is the number of connected components of $(\partial \mr M-\partial N_\epsilon)\cap B_J(j)$.

We can use the coorientation of $\partial\hbs$ to totally order all the conduits of $S^\flat$ meeting $J$ on the same side as $c_1$ and $c_2$. If there are no conduits of $\partial \mr S^\flat$ lying between $c_1$ and $c_2$ with respect to this order, we say the cusp $j$ is \textbf{free}.

\begin{figure}
\centering
\includegraphics[height=1.5in]{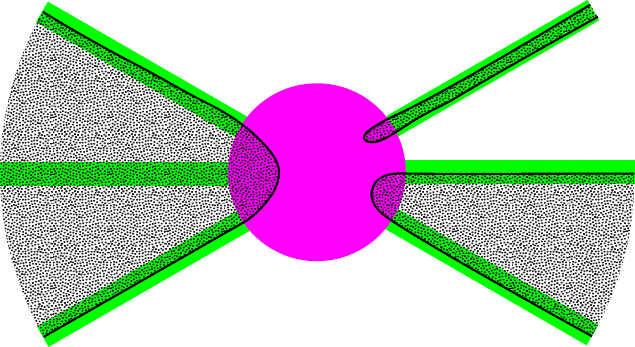}
\caption{We have drawn a junction (magenta) and six conduits (green) of $\partial N_\epsilon$. There are three cusps shown, all of which are free. On the left is a cusp of size 2. The top right cusp has size 0 and the bottom right cusp has size 1. The shading indicates the $\flat$-disk in whose boundary each cusp lies.}
\label{fig_cusp}
\end{figure}

\subsection{Flat isotopy}\label{sec_flatisotopy}

Similar to our drawings of $\partial M$, when drawing a flattened surface we omit drawing actual plates and rods of $N_\epsilon$ and simply draw the surface near $\tau^{(2)}$ in such a way that it is clear where $S^\flat$ lies in $N_\epsilon$.

Let $S^\flat$ be a flattened surface, whose identity will change by isotopies. We will now describe a series of possible isotopies we can perform on $S^\flat$. Our goal, ultimately, is to perform sequences of these moves to minimize the area of $S^\flat$. We include pictures of these to aid the reader, but to reduce clutter we only draw the parts of the isotopy that take place in $\mr M$.
\begin{figure}
\centering
\includegraphics[height=1.25in]{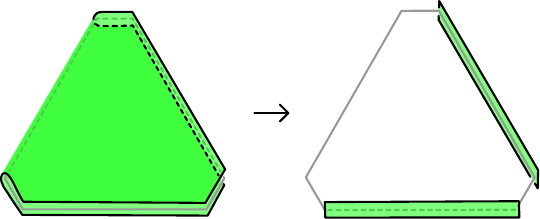}
\caption{Face move: a size 0 cusp of an innermost $\flat$-disk for $S^\flat$ corresponds to a portion of $S^\flat$ running back and forth over a single $\tau$-face. This can be eliminated by an isotopy of $S^\flat$.}
\label{fig_facemove}
\end{figure}

\medskip

\textbf{Face move.} Whenever $\partial \mr S^\flat$ has two conduit segments $c_1$ and $c_2$ connected by a junction segment $j$ such that $c_1$ is positive, $c_2$ is negative, and both $c_i$ lie in the same conduit, we say that the concatenation $c_1*j*c_2$ is \textbf{backtracking}. This corresponds to a portion of $S^\flat$ running across a plate corresponding to a $\tau$-face $\hex$, into a rod and back across the same plate. If there are no other conduit segments between $c_1$ and $c_2$, then this portion of $S^\flat$ can be eliminated by an isotopy of $S^\flat$ as in \Cref{fig_facemove}. This induces an isotopy of $\partial \mr S^\flat$ at two tips of the corresponding face $\hex$, and a cut and paste surgery of $\partial \mr S^\flat$ at the third tip of $\hex$. 
Let $U_1$ be the component of $U$ corresponding to this third tip. The cut and paste surgery is a result of pushing a small piece of $S^\flat$ which used to lie in $\mr M$ into $U_1$. If this creates a component of $S^\flat\cap U_1$ which is not $\pi_1$-injective in $U$, we can use the irreducibility of $M$ and incompressibility of $S^\flat$ to correct this via an isotopy. The initial isotopy combined with the second (if necessary) is called a \textbf{face move}. We see that whenever $\partial \mr S^\flat$ has backtracking, the area of $S^\flat$ can be reduced by at least 2 by a face move.

\medskip

\textbf{Disk removals.} Let $\delta$ be an innermost $\flat$-disk. If $\delta$ has volume 1 and circumference 3 or volume 0 and circumference 2, then $\delta$ can be eliminated by an isotopy of $S^\flat$ so that $S^\flat\cap \mr M$ still lies in $N_\epsilon$ transverse to the vertical foliation. Each of these decreases the area of $S^\flat$ by 2. See \Cref{fig_removal}.

\medskip

\begin{figure}
\centering
\includegraphics[height=2.5in]{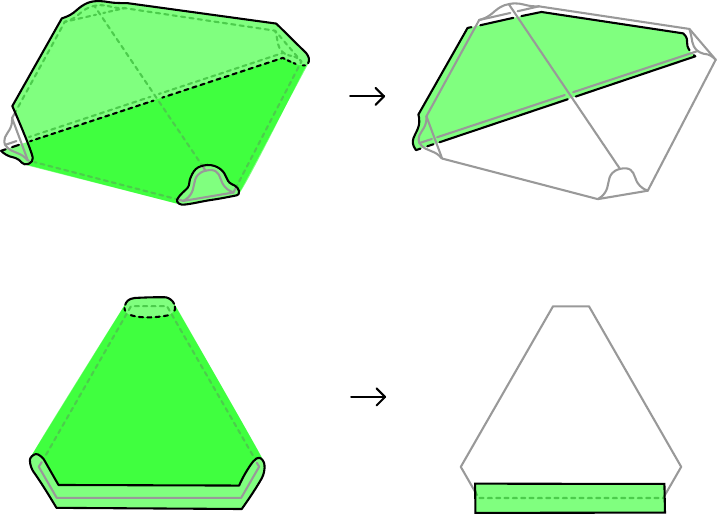}
\caption{Top: disk removal of a $\flat$-disk with volume 1 and circumference 3. Bottom: disk removal of a $\flat$-disk with volume 0 and circumference 2.}
\label{fig_removal}
\end{figure}

\textbf{Tetrahedron move.} Let $\delta$ be a $\flat$-disk, not necessarily innermost.
If $c$ is a free cusp of $\delta$ with size 1, then there is an isotopy of $S^\flat$ as shown in \Cref{fig_tetmove} which sweeps a portion of $S^\flat$ across a $\tau$-tetrahedron. This move decreases the volume of $\delta$, and if $\delta$ is inward (outward) it does not increase the volume of any inward (outward) $\flat$-disk.

\begin{figure}
\centering
\includegraphics[height=1.25in]{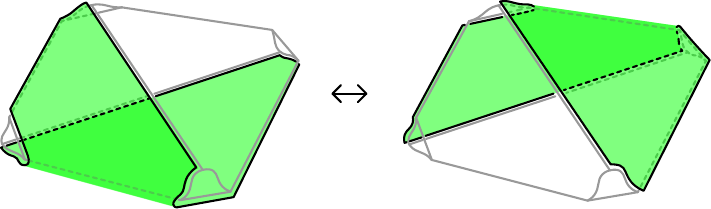}
\caption{Tetrahedron move.}
\label{fig_tetmove}
\end{figure}

\begin{figure}
\centering
\includegraphics[height=1.25in]{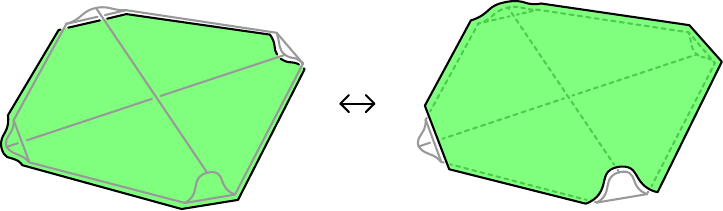}
\caption{Flip move.}
\label{fig_flip}
\end{figure}

\medskip

\textbf{Flip move.} If $S^\flat$ has two adjacent plate regions $P$ and $Q$ corresponding to the bottom of a single $\tau$-tetrahedron $\tet$, and $P$ and $Q$ are topmost in their respective plates of $N_\epsilon$, then there is an isotopy of $S^\flat$ which sweeps a portion of $S$ up through $\tet$ and replaces $P$ and $Q$ by two bottommost plate regions $P'$ and $Q'$ in the plates corresponding to the top of $\tet$. This is called an \textbf{upward flip move}. The inverse of an upward flip is called a \textbf{downward flip move}.

\medskip

\begin{figure}
\centering
\includegraphics[height=2in]{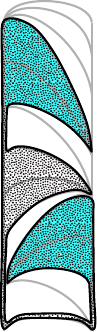}
\caption{A $\flat$-disk $\delta$ with no backtracking decomposes into fat regions connected by thin regions, giving a finite tree. If $\delta$ has width one, there is an available flip move or tetrahedron move in each region corresponding to a vertex of valence 1 in this tree. In the picture above, the tree is a path with three vertices, and in each region corresponding to a valence 1 vertex (colored cyan) there is an available downward flip move which shrinks the $\flat$-disk.}
\label{fig_widthonebad}
\end{figure}

\textbf{Width one move.} If $\delta$ is an innermost $\flat$-disk of width 1, then the area of $S^\flat$ can be reduced by a combination of face moves, tetrahedron moves, flip moves, and disk removals applied to $\del\delta$. If $\delta$ has backtracking, this is clear: we simply apply a face move.
Otherwise $\delta$ has no backtracking. In this case, $\delta$ naturally decomposes into a finite number of \emph{fat} regions with positive volume separated by \emph{thin} regions entirely contained in $\del N_\epsilon$. For an example, see \Cref{fig_widthonebad}. We can define a graph $G_\delta$ whose vertices correspond to the fat regions and whose edges correspond to the thin regions. Since $\delta$ is a disk, the graph $G_\delta$ is a tree. 

Suppose that $\delta$ lives in an upward ladder (the downward case is symmetric). If $G_\delta$ consists of a single vertex, then $\delta$ has a free cusp of size one at the lowest flat triangle in $\delta$ so the area of $\delta$ may be reduced by a tetrahedron move. In addition, $\del\delta$ runs through both top conduits of the highest flat triangle in $\delta$, and therefore a downward flip move will also reduce the volume of $\delta$.
If $G_\delta$ is a more complicated tree, let $F$ be a fat region corresponding to a $G_\delta$-vertex of degree 1. Then, depending on where $F$ meets a thin region, there is either an available downward flip move at the top flat triangle in $F$ or a tetrahedron move at the bottom flat triangle in $F$. The move reduces the volume of $\delta$. Since $\delta$ has finite volume, finitely many moves suffice to either introduce backtracking or reduce the volume of $\delta$ to 1. At this point a face move or disk removal reduces the area of $S^\flat$.
A \textbf{width one move} is defined to be any such area-reducing sequence of moves applied to a width one $\flat$-disk.

\begin{figure}
\centering
\includegraphics[width=5in]{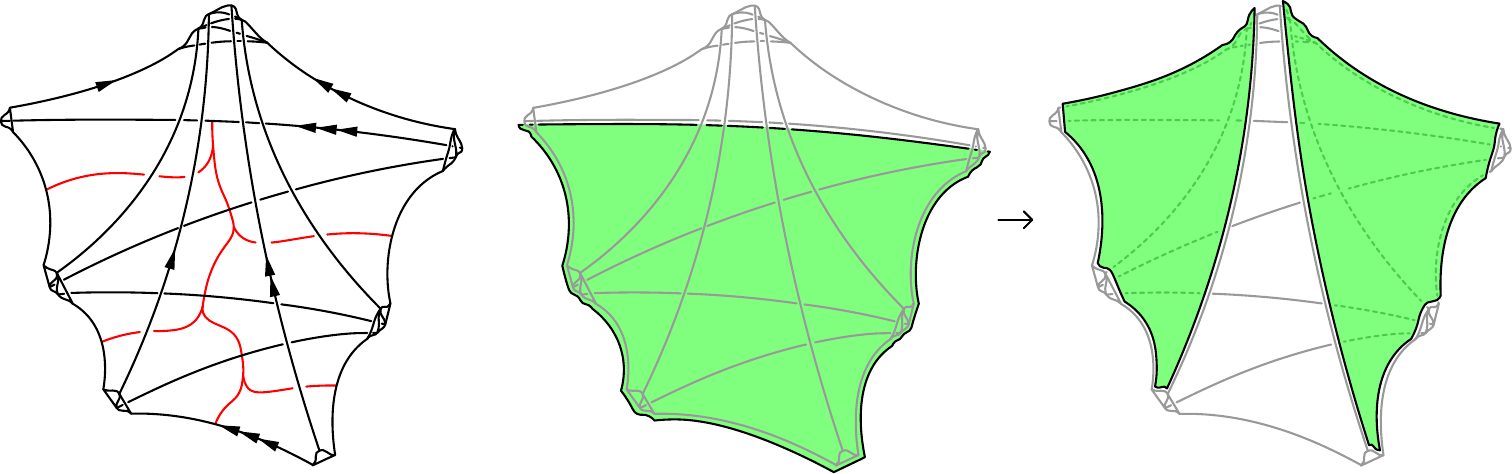}
\caption{\emph{An example of a stable annulus move}. Left: a portion of a veering triangulation. Some edge identifications are specified, but we make no assumptions about the others. The portion of $B^s\cap \tau^{(2)}$ shown carries a shallow stable loop. Center: An example of an annulus satisfying the conditions for a stable annulus move. Right: the situation after the stable annulus move. The core of the annulus has been pushed into the component of $U$ corresponding to the topmost (in the page) tips of the four $\tau$-tetrahedra.}
\label{fig_annulusmove}
\end{figure}

\medskip

\textbf{Annulus move.} Our last move is different than the previous ones in that its nature is more global than local. Suppose there is an annulus $A\subset S^\flat$ with the following properties: 
\begin{enumerate}[label=(\roman*)]
\item $A$ is a union of rods and plates,
\item each plate of $A$ is topmost in its $N_\epsilon$-plate, and
\item the induced stable train track on $A$ carries a curve whose image under collapsing is a shallow stable loop $\gamma$. 
\end{enumerate}
Then we can perform a \textbf{stable annulus move}, which we now describe. 

Let $\wt\gamma$ be a lift of $\gamma$ to $\wt M^\circ$. Since $\gamma$ is a shallow stable loop, \Cref{lem_shallowloops} gives a ladder $L\subset \partial \mr M$ and a lift $\wt L$ of $L$ to $\wt M^\circ$ such that, if $S(\wt L)$ is the union of tetrahedra intersecting $\intr(\wt L)$, the faces traversed by $\wt\gamma$ are exactly the faces of $S(\wt L)$ which do not meet $L$. The image $T$ of $S(\wt L)$ under the covering map is a possibly immersed solid torus in $\mr M$. Its interior is embedded since otherwise the annulus $A$ would necessarily intersect itself, contradicting that $S^\flat$ is embedded.

It follows that there exists an isotopy of $S^\flat$ which sweeps $A$ up through $T$, pushing a core subannulus of $A$ through the ladder $L$ into the component of $U$ corresponding to $L$, so that after this isotopy the remainder of $A$ lies in the plates and rods corresponding to the two ladderpoles of $L$. See \Cref{fig_annulusmove}. Note that the new plates of $S^\flat$ after the annulus move will have the same sign as the plates of $A$.

An \textbf{unstable annulus move} is defined symmetrically. For this move we require each plate region of the annulus $A$ to be bottommost in its plate, and we require that the induced unstable train track carry a curve which collapses to a shallow unstable loop.

%%%%%%%%%%%%%%%%%%%%%%%%%%%%%%%%%%%%%%%%%%%%%%%%%%%%%%%%%%%%%%%%%%%%%%%%%%%%%%%%%%%%%%%%%%%%%%%%%%%%%%%%%%%%%%%%%%
\section{$\C_\tau^\vee\subset \cone(\sigma_\tau)$}\label{sec_easycontainment}

Recall that $\Gamma$ denotes the dual graph of $\tau$, which is the 1-skeleton of the 2-complex dual to $\tau$.
Let $\C_\tau\subset H_1(M)$, the \textbf{cone of homology directions} of $\tau$, be the smallest convex cone in $H_1(M)$ containing the homology classes of oriented cycles in $\Gamma$. Associated to $\C_\tau$ is a dual cone in $H_2(M)$ defined by 
\[
\C_\tau^\vee=\{\alpha\in H_2(M)\mid \langle \alpha, \gamma\rangle \ge 0 \text{  for all  } \gamma\in \C_\tau\}.
\]

Let $U_i$ be a component of $U$ and let $\prongs(U_i)$ be the number of prongs of the meridional slope of $U_i$. The component $R_i$ of $M-B^s$ containing $U_i$ is homeomorphic to an open solid torus. \Cref{lem_tipcondition} gives us that its closure $C_i$ in the path topology has the smooth structure of the mapping torus of some rotation of a $(\prongs(U_i))$-gon patch of a train track. We call $R$ a \textbf{tube} of $B^s$. Similarly $U_i$ determines a complementary region of $B^u$ with the same type of structure, which we call a \textbf{tube} of $B^u$.
The \textbf{index} of $U_i$ is the number
\[
\ind(U_i)=2-\prongs(U_i),
\]
which is equal to the minimal index of a meridional disk of $C_i$. Since $\prongs(U_i)\ge 3$, we have $\ind(U_i)\le -1$.

Using this language, if $S\subset M$ is a surface, then a patch of $S\cap B^{s/u}$ is simple if it $\pi_1$-injects into its tube.

\begin{lemma}[no nullgons or monogons]\label{lem_nonullgons}
Let $S$ be an incompressible surface in $M$. Then $S$ is isotopic to a flattened surface $S^\flat$ such that all patches of both $S^\flat \cap B^s$ and $S^\flat \cap B^u$ are simple and have nonpositive index.
\end{lemma}

\begin{proof}
We first isotope  $S$ so that it has simple patches, which is possible by \Cref{lem_simplepatches}. Then we flatten $S$ to a surface $S^\flat$. Let $t^s$ and $t^u$ denote $S^\flat\cap B^s$ and $S^\flat\cap B^u$. Any meridional disk patch of $t^s$ or $t^u$ has index $\le -1$ because each component of $U$ has index $\le -1$.
By applying face moves and width one moves to decrease the area of $S^\flat$,
we can isotope  $S^\flat$ so that it has no $\flat$-disks of width one or zero. 
Suppose that $p$ is a nonmeridional disk patch of $t^s$. If $\delta$ has width $\ge 2$, $\partial\delta$ must contain at least two rung segments crossing an upward ladder, so by \Cref{lem_tipcondition} $p$ must have at least 2 switches. Symmetric reasoning shows the same holds for patches of $t^u$.

In summary, every disk patch has nonpositive index. Since topological annulus patches must have nonpositive index, the proof is complete.
\end{proof}

\begin{lemma}[$M$ is atoroidal]\label{lem_atoroidal}
There is no incompressible torus embedded in $M$.
\end{lemma}

\begin{proof}
Suppose for a contradiction that $T$ is an incompressible torus embedded in $M$. By \Cref{lem_nonullgons}, $T$ can be flattened to $T^\flat$  such that $T^\flat \cap B^s$ has no patches of positive index. Since $\chi(T)=0$, all patches of $T^\flat \cap B^s$ must have index 0. Therefore, since the index of each tube is at least $-1$, $T^\flat$ must not intersect any tube in a meridional disk. As such we may isotope  $T^\flat$ so that it lies in $\mr M$.
By \cite[Theorem 1.5]{HRST11} $\tau$ admits a strict angle structure and by a result of Casson exposited in \cite[Theorem 1.1]{FutGue11}, $\intr(\mr M)$ is irreducible and atoroidal. In particular $T^\flat$ must be either peripheral in $\mr M$, in which case $T$ is compressible in $M$, a contradiction; or compressible in $\mr M$, in which case $T$ is also compressible in $M$. \end{proof}

It follows that we can speak of the Thurston norm $x$ and the dual norm $x^*$ as norms on $H_2(M)$ and $H_1(M)$, respectively. The \textbf{Euler class of $\boldsymbol\tau$} is the homology class in $H_1(M)$ defined by
\[
e_\tau=\frac{1}{2}\cdot\sum \ind(U_i)\cdot [\core(U_i)],
\]
where the sum is over components of $U$ and $\core(U_i)$ denotes the core curve of $U_i$, oriented so that a curve in $\partial U_i$ lying in a single ladder positively transverse to its rungs is homotopic to a positive multiple of $\core(U_i)$ . We define
\[
E=\{\alpha\in H_2(M)\mid \langle -e_\tau, \alpha \rangle=x(\alpha)\}.
\]
By definition, if $E$ is strictly larger than $\{0\}$ then it is a cone over some face of $B_x(M)$. We denote this face by $\sigma_\tau$. If $E=\{0\}$ we let $\sigma_\tau$ be the empty face of $B_x(M)$.

\begin{lemma}\label{lem_dualnorm}
The Euler class of $\tau$ has dual Thurston norm at most 1:
\[
x^*(e_\tau)\le 1.
\]
\end{lemma}
\begin{proof}
It suffices to show that for any taut surface $S$ we have $|\langle e_\tau,[S]\rangle |\le \chi_-(S)$ or alternatively that $\langle e_\tau,[S]\rangle\ge \chi(S)$. First, we can flatten $S$ to $S^\flat$ and eliminate $\flat$-disks of width $\le 1$, guaranteeing that no patches of $S^\flat\cap B^s$ have positive index (the same is true for $S^\flat\cap B^u$ but we do not need that here). The only patches which have nonzero intersection with $\bigcup_i\core(U_i)$ are the meridional disk patches. If $C$ is a meridional disk patch, let $\sgn(C)=\pm1$ according to whether the algebraic intersection of $C$ with the core of the corresponding component of $U(C)$ of $U$ is $\pm1$. Then
\[
2\langle e_\tau,[S]\rangle=\sum_{\text{merid}} \sgn(C)\cdot \ind(U(C)),
\]
where the sum is over only meridional patches. Since all patches have nonpositive index,
\[
\sum_{\text{merid}} \sgn(C)\cdot \ind(U(C))\ge \sum_{\text{merid}}  \ind(C) \ge \sum_{\text{all patches}} \ind(P)=2\chi(S^\flat)=2\chi(S),
\]
where the final sum is over all patches of $S^\flat \cap B^s$.
\end{proof}

\begin{lemma}\label{lem_ifcarriedthentaut}
If $S$ is carried by $\tau^{(2)}$ (as a partial branched surface) then $S$ is taut and $[S]\in \cone(\sigma_\tau)$.
\end{lemma}
\begin{proof}
By \cite[Lemma 5.8]{LMT20} the dual graph $\Gamma$ is \emph{strongly connected}, i.e. for any $\Gamma$-vertices $u$ and $v$, there exists a directed path in $\Gamma$ from $u$ to $v$. This implies that for each $\tau$-face $\hex$, there is an closed oriented curve through $\hex$ which is positively transverse to $\hbs$. If $S$ is carried by $\hbs$, and $T$ is any component of $S$, it follows that we can find a closed oriented curve whose intersections with $T$ are all positive. Therefore $S$ has no nulhomologous components.

Next we show that $\chi_-(S)=x([S])$. We have $x([S])\le \chi_-(S)$ by the definition of $x$, and by the above lemma we have that $x^*(-e_\tau)\le 1$. If $S$ is carried by $\tau^{(2)}$ then each patch of $S^\flat\cap B^s$ intersecting a component $u$ of $U$ corresponds to either a meridional curve on $\partial U$ which traverses exactly $\prongs(u)$ rungs of upward ladders, or to 2 ladderpole curves crossing no rungs. It follows that each patch of $S\cap B^s$ with nonzero index is a meridional disk patch with index equal to the index of the corresponding tube. Hence $\langle e_\tau, [S] \rangle=\chi(S)$. Therefore
\[
\chi_-(S)=\langle -e_\tau,[S]\rangle\le x([S]),
\]
completing the proof.
\end{proof}

\begin{theorem}\label{prop_carriedcriterion}
Let $\alpha\in H_2(M)$ be an integral class. Then $\alpha\in  \C_\tau^\vee$ if and only if there exists a surface $S$, necessarily taut, carried by $\hbs$ with $[S]=\alpha$.
\end{theorem}
\begin{proof}
If $S$ is carried by $\tau$ and represents $\alpha$, it is clear that $\alpha\in \C_\tau^\vee$.

Conversely, suppose $\alpha\in \C_\tau^\vee$. Let $u\in H^1(M)$ be the Poincare dual of $\alpha$. Note that $u$ pulls back to a cohomology class $\mr u$ on $\mr M$ which takes integral values on directed cycles in $\Gamma$. Let $\mr \alpha\in H_2(M,\partial M)$ be the Lefschetz dual of $\mr u$. As explained in \cite[Proposition 5.11]{LMT20},
$\alpha$ is represented by a surface $\mr S$ which is carried by $\tau^{(2)}$, viewed as a branched surface in $\mr M$. 

The proof of \cite[Lemma 3.3]{Lan18} goes through in this setting and thus, for each component $\partial U_i$ of $\partial \mr M$, $S\cap \partial U_i$ is either (i) a nulhomologous collection of curves of ladderpole slope on $\partial U_i$, (ii) a collection of curves with the meridional slope of $U_i$, or (iii) empty. Hence $\mr S$ can be capped off to give a surface $S\subset M$ which is carried by $\tau^{(2)}$. 
The composition $P$ of the maps 
\[
H_2(M)\xrightarrow{PD} H^1(M)\xrightarrow{i^*} H^1(\mr M)\xrightarrow{LD} H_2(\mr M,\del\mr M),
\] 
where the first and last arrows are Poincar\'e and Lefschetz duality and the middle arrow is pullback under inclusion, is injective as explained in \cite[\S 2.1]{Lan18} (in that paper it is called the \emph{puncturing map}). We have $P(\alpha)=\mr \alpha=P([S])$, so $[S]=\alpha$. 
By \Cref{lem_ifcarriedthentaut}, $S$ is taut.
\end{proof}

\begin{corollary}\label{lem_easycontain}
$\C_\tau^\vee\subset \cone(\sigma_\tau)$.
\end{corollary}
\begin{proof}
If $\alpha\in \C_\tau^\vee$ is an integral class then \Cref{prop_carriedcriterion} furnishes a surface representing $\alpha$ carried by $\hbs$, so $\alpha\in \cone(\sigma_\tau)$ by \Cref{lem_ifcarriedthentaut}.
\end{proof}

%%%%%%%%%%%%%%%%%%%%%%%%%%%%%%%%%%%%%%%%%%%%%%%%%%%%%%%%%%%%%%%%%%%%%%%%%%%%%%%%%%%%%%%%%%%%%%%%%%%%%%%%%%%%%%%%%%

\section{The bigon property and efficient bigon property}\label{sec_bigprops}

\subsection{The bigon property} \label{sec_bigonproperty}

Let $S$ be an embedded surface in $M$ transverse to $B^s$ and $B^u$. Each switch of $S\cap B^s$ corresponds to an intersection of $S$ with $\Gamma$, and such an intersection point can be either positively or negatively oriented according to whether the orientation of $\Gamma$ agrees with the coorientation of $S$ at the intersection point. Switches corresponding to positive (negative) intersections will be called \textbf{positive} (\textbf{negative}) switches.

Suppose that
\begin{enumerate}[label=(\alph*)]
\item no component of $S$ is nulhomologous,
\item $S$ has simple patches with respect to $B^s$,
\item each patch of $S\cap B^s$ has nonpositive index, and
\item each negative switch of $S\cap B^s$ belongs to a bigon.
\end{enumerate}
Then we say $S$ has the \textbf{stable bigon property}. If $S$ has the above properties with respect to $B^u$, then $S$ has the \textbf{unstable bigon property}. If $S$ has both the stable and unstable bigon properties, we say $S$ has the \textbf{bigon property}. The bigon property will provide us with traction as we try to isotope  surfaces to be carried by $\hbs$. Although we will not need to use the following fact, the bigon property has topological and algebraic consequences:

\begin{lemma}\label{lem_bigontaut}
Let $S$ be a surface with the stable or unstable bigon property. Then $S$ is taut and $[S]\in\cone(\sigma_\tau)$.
\end{lemma}

\begin{proof}
We break symmetry and suppose that $S$ has the stable bigon property. Then the only patches of $S\cap B^s$ which possibly have nonzero index are meridional disk and topological annulus patches. It is impossible for a topological annulus patch of $S\cap B^s$ to have a positive switch without also having a negative switch (this follows from the structure of the tubes of $B^s$ and $B^u$, see discussion at the beginning of \Cref{sec_easycontainment}), so the topological annulus patches must be annulus patches i.e. must have no switches. Hence $2\chi(S)$ is the sum of the indices of its meridional patches. Because the switches of these patches are all positive, each one contributes $+1$ to the algebraic intersection of $[S]$ with the core of the tube. Hence $2\chi(S)$ is equal to $2\langle e_\tau,[S]\rangle$, so
\[
\chi_-(S)=\langle -e_\tau,[S]\rangle\le x([S])\le\chi_-(S),
\]
using the fact that $x^*(-e_\tau)\le1$ by \Cref{lem_dualnorm}. Since no component of $S$ is nulhomologous, $S$ is taut.
\end{proof}

\begin{lemma}\label{lem_indexsumarg}
Let $S$ be a taut surface in $M$ with $[S]\in \cone(\sigma_\tau)$. Then $S$ is isotopic to a flattened surface with the bigon property.
\end{lemma}

\begin{proof}
By \Cref{lem_nonullgons}, we can flatten $S$ to $S^\flat$ so that $S^\flat$ has simple patches with nonpositive index. We will show that $S^\flat$ has the stable bigon property; the proof that $S^\flat$ has the unstable bigon property is entirely similar.

To prove $S^\flat$ has the stable bigon property it suffices to show that any patch of $S^\flat\cap B^s$ which is not a bigon has only positive switches. 

For a tube $T$ of $B^s$, define $\ind(T)$ to be the index of the component of $U$ at the core of $T$. 
Just as in the proof of \Cref{lem_dualnorm}, the patches of $S^\flat$ are all topologically disks or annuli and of these, the only patches having nonzero algebraic intersection with the core of the corresponding tube are meridional disks. 
Therefore the definition of $e_\tau$ gives us
\[
2\chi(S^\flat)=-2x([S^\flat])=2e_\phi([S^\flat])=\sum_{\text{tubes $T$}}\left(\sum_{\text{merid. disks $D$}}\ind(T)\cdot\text{sign}(D)\right),
\]
where $\text{sign}(D)=\pm1$ is the algebraic intersection of $D$ with the core of the corresponding tube. Subtracting \Cref{indexeq} from the above, we have
\[
0=\sum_{\text{tubes $T$}}\left(\sum_{\text{merid. disks $D$}}\ind(T)\cdot\text{sign}(D)-\ind(D)\right) -\sum_{\text{nonmerid. $C$}}\ind(C),
\]
so
\[
\sum_{C}\ind(C)=\sum_{T}\left(\sum_{D}\ind(T)\cdot\text{sign}(D)-\ind(D)\right).
\]
Every term of the lefthand sum is nonpositive since there are no monogons or nullgons. Moreover every term of the righthand sum is nonnegative. For disks with $\sgn(D)=-1$ this is trivial and for disks $D$ with $\sgn(D)=1$ this follows from the fact that any meridional disk of a tube $T$ has index bounded above by the index of the tube in which it sits.

We conclude that each term is zero. It follows that every nonmeridional patch has index 0 and is therefore an annulus or bigon. Each meridional patch has sign 1 and index equal to the index of the corresponding tube, so each switch of a non-bigon patch must be positive, whence $S$ has the stable bigon property. An identical argument shows that $S^\flat$ has the unstable bigon property.
\end{proof}

\begin{lemma}\label{lem_negativerungs}
Let $S^\flat$ be a flattened surface with the bigon property and let $\gamma$ be a component of $S^\flat\cap \partial\mr M$. Then the following are equivalent:
\begin{enumerate}[label=(\roman*)]
\item$\gamma$ contains a negative rung conduit,
\item $\gamma$ bounds a $\flat$-disk, and
\item the patches of $S^\flat\cap B^s$ and $S^\flat\cap B^u$ containing $\gamma$ are both bigons.
\end{enumerate}
\end{lemma}
\begin{proof}
This follows from \Cref{lem_tipcondition} and the definition of the bigon property. 
\end{proof}

\begin{lemma}\label{lem_diskproperties}
Let $S^\flat$ be a flattened surface with the bigon property. Each $\flat$-disk $\delta$ of $S^\flat$ has width two and has exactly four rung conduit segments.
If in addition $S^\flat$ has a $\flat$-disk with two adjacent rung conduit segments meeting at a free cusp of size 1, then $S^\flat$ is flat isotopic to a surface with smaller area. 
\end{lemma}

\begin{proof}
Let $t^s=S^\flat\cap B^s$ and $t^u=S^\flat\cap B^u$. The $\flat$-disk $\delta$ corresponds to a bigon patch of $t^s$, so $\partial\delta$ must contain exactly two rung conduit segments crossing upward ladders by \Cref{lem_tipcondition}. Similarly $\delta$ corresponds to a bigon patch of $t^u$, so $\del\delta$ must contain exactly two rung conduit segments crossing downward ladders.

For the last claim, note that performing the available tetrahedron move on the free cusp creates a $\flat$-disk of width one. Applying a subsequent width one move decreases the area of $S^\flat$.
\end{proof}

\begin{corollary}\label{lem_nohingemoves}
Suppose $S^\flat$ is a flattened surface with the bigon property. Suppose that $S^\flat$ has two plates $P_1$ and $P_2$ separated by a single rod such that there exists a hinge tetrahedron $\tet$ which is immediately above $P_1$ and below $P_2$. Then $S^\flat$ is flat isotopic to a surface with smaller area.
\end{corollary}
\begin{proof}
There is a unique tip $\tri$ of $\tet$ such that the corresponding tips $t_1$ and $t_2$ of $P_1$ and $P_2$ are both rung conduit segments (see \Cref{fig_hinge}). The signs of $t_1$ and $t_2$ are opposite, so  $t_1$ and $t_2$ must lie in the boundary of a $\flat$-disk by \Cref{lem_negativerungs}. Since each $\flat$-disk has width two and four rung conduit segments in its boundary, $t_1$ and $t_2$ must meet at a free cusp of size 1. Now we apply the last sentence of \Cref{lem_diskproperties}.
\end{proof}

\begin{figure}
\centering
\includegraphics[width=5in]{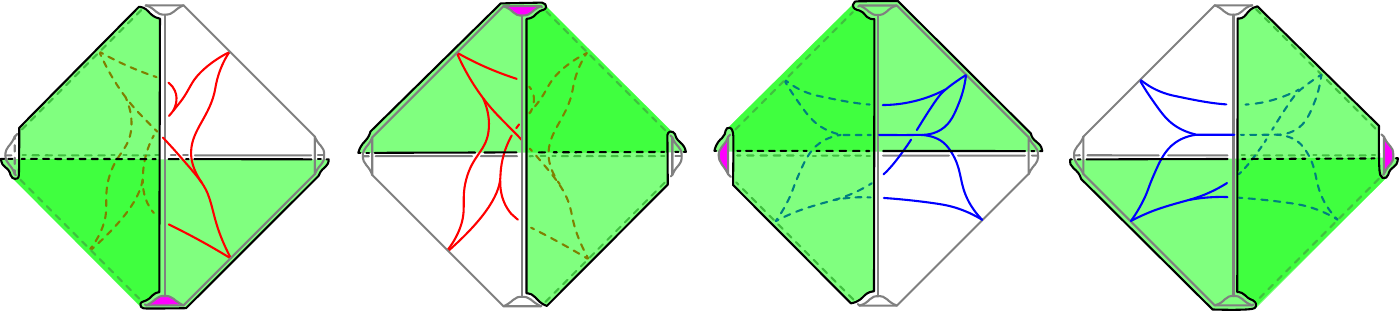}
\caption[]{Here we have drawn the same hinge tetrahedron $\tet$ four times, where the top $\tau$-edge of $\tet$ is right veering and the bottom $\tau$-edge is left veering. For any possible configuration of $P_1$ and $P_2$ as in the statement of \Cref{lem_nohingemoves}, we have colored the tip $\tri$ (using notation from the proof) of $\tet$ magenta. } 
\label{fig_hinge}
\end{figure}

\begin{lemma}\label{lem_residualbp}
Suppose that $S^\flat$ is a flattened surface with the bigon property and that $S^\flat$ has a $\flat$-disk with a free cusp of size 1. After performing the associated tetrahedron move, $S^\flat$ either still has the bigon property or is flat isotopic to a surface with smaller area.
\end{lemma}

\begin{proof}
A tetrahedron move preserves the topology of $S^\flat$, and thus creates no nulhomologous components. It also preserves the underlying topology of patches of $S^\flat\cap B^s$ and $S^\flat \cap B^u$, and thus preserves the property of having simple patches. 
Accordingly, if $S^\flat$ does not have the bigon property after a tetrahedron move then it must be the case that either (i) the move creates a patch of positive index or (ii) the move creates a meridional disk patch or topological annulus patch of $S^\flat\cap B^s$ or $S^\flat \cap B^u$ with at least one negative switch.

If (i) occurs, note that the patch $p$ of positive index must be a nullgon or monogon. Since a $\flat$-disk cannot contain only one rung conduit, by \Cref{lem_tipcondition} $p$ must be a nullgon and correspond to a $\flat$-disk of width $1$. Now applying a width one move decreases the area of $S^\flat$.

If (ii) occurs, let $p$ be the meridional disk or annulus patch with a negative switch. We claim the index of $p$ is strictly less than it was prior to the tetrahedron move. Indeed:
\begin{itemize}
\item If $p$ is a meridional disk, then prior to the tetrahedron move the bigon property guaranteed that all its switches were positive. A meridional disk in a tube of $B^s$ or $B^u$ has maximal index over all such meridional disks exactly when its switches have identical sign, so the index of $p$ must have decreased.
\item If $p$ is an annulus, then as in the proof of \Cref{lem_bigontaut} the definition of the bigon property implies that prior to the tetrahedron move $p$ had no switches. Therefore the existence of a negative switch means that the index of $p$ must have decreased.
\end{itemize}

Since the sum of indices over all patches remains unchanged, the index of some patch $q$ must have increased. Since $S^\flat$ initially had the bigon property, the indices of meridional patches and annulus patches cannot have increased, so $q$ must be a nonmeridional disk patch. Therefore $q$ must have positive index,  so we are back in case (i).
\end{proof}

\subsection{The efficient bigon property}

Let $S^\flat$ be a flattened surface with the bigon property. There are limitations on the types of $\flat$-disks that $S^\flat$ can have, described in \Cref{lem_diskproperties}: the $\flat$-disks of $S^\flat$ all have width 2 and their boundaries contain four rungs. We wish to now simplify this picture by placing further restrictions on $\flat$-disks.

 If $\delta$ is a $\flat$-disk of $S^\flat$ such that negative conduits are unlinked from positive conduits in $\partial \delta$, we say $\delta$ is \textbf{efficient}. The obstruction to the efficiency of $\delta$ is the existence of a \textbf{kink} in $\partial\delta$, which is a portion $k$ of $\partial\delta$ containing only ladderpole conduits of the same sign $s$ such that the conduits to either side of $k$ are of the same sign $s'$ with $s\ne s'$. See \Cref{fig_kinkshadow}.

Let $\delta$ be a $\flat$-disk of $S^\flat$ and let $r_1$ and $r_2$ be two rung conduits of $\partial \delta$ contained in the same ladder $L$. If there is an oriented transversal to $\partial\hbs$ contained in $L\cap \delta$ starting at $r_1$ and ending at $r_2$, we say $r_1$ is a \textbf{lower rung} for $\partial\delta$ and $r_2$ is an \textbf{upper rung} for $\partial\delta$.
A kink $\kappa$ always connects two upper rungs or two lower rungs. If $\kappa$ connects two upper or lower rungs, we say $\kappa$ is an \textbf{upper} or \textbf{lower kink}, respectively.
The \textbf{length} of a kink $k$ is the number of conduits it contains. 

Let $\kappa$ be a kink in $\partial\delta$ for a $\flat$-disk $\delta$ of $S^\flat$. Then $\kappa$ determines a region of $\delta$ called the \textbf{shadow} of $\kappa$, which we now define. If $\kappa$ is an upper (lower) kink, let $L$ be downward (upward) ladder incident to $\kappa$. Let $v$ be the lowest (highest) junction traversed by $\kappa$. Let $r$ be the upper (lower) rung in $\partial\delta$ crossing $L$, and let $r'$ be the topmost (bottommost) rung in $L$ which is incident to $v$. The region in $L$ bounded above (below) by $r$ and below (above) by $r'$ is called the \textbf{shadow of $\kappa$} and denoted $\shad(\kappa)$.
See \Cref{fig_kinkshadow}.

\begin{figure}
\centering
\includegraphics[]{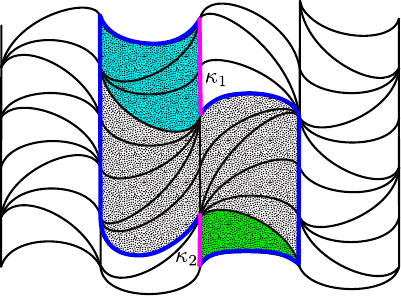}
\caption{An inefficient $\flat$-disk with 2 kinks shown in magenta: $\kappa_1$ is an upper kink of length 2 and $\kappa_2$ is a lower kink of length 1. In cyan we see $\shad(\kappa_1)$, and $\shad(\kappa_2)$ is shown in green.}
\label{fig_kinkshadow}
\end{figure}

\begin{definition}
Let $S^\flat$ be a flattened surface with the bigon property. We say that $S^\flat$ has the \textbf{efficient bigon property} if $\partial\mr S^\flat$ has no backtracking and every $\flat$-disk of $S^\flat$ is efficient.
\end{definition}

Our next goal is to show that by applying flat isotopies to decrease the area of $S^\flat$, we can assume that it has the efficient bigon property. 

\begin{lemma}\label{lem_shadow}
Let $S^\flat$ be a flattened surface with the bigon property. Let $\delta$ be an inefficient $\flat$-disk with kink $\kappa$. If the shadow of $\kappa$ contains a hinge flat triangle, then $S^\flat$ is flat isotopic to a surface with smaller area.
\end{lemma}
\begin{proof}
There is a portion of $\del\delta$ that, omitting junctions, can be expressed as a concatenation $(u,\kappa, v ,s)$ where $u$ and $v$ are rung conduits of the same sign, and there is a cusp at the junction of $v$ with $\kappa$ and $s$ ($s$ could be rung or ladderpole in this discussion). Let $L_u$ and $L_v$ be the ladders crossed by $u$ and $v$ respectively. We assume that $\kappa$ is an upper kink; the argument when $\kappa$ is a lower kink is symmetric. We have drawn this picture in \Cref{fig_efficient12}. 
\begin{figure}
\centering
\includegraphics[]{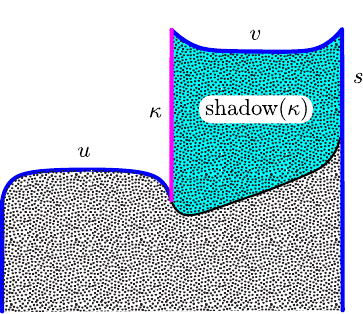}
\caption{Notation from \Cref{lem_shadow}.}
\label{fig_efficient12}
\end{figure}

Suppose that $\shad(\kappa)$ contains a hinge flat triangle $h$. 
If $\delta$ is innermost, by applying tetrahedron moves we can begin shrinking $A$ one flat triangle at a time. After each move, by \Cref{lem_residualbp} $S^\flat$ either has the bigon property or has reducible area. If all of these moves preserve the bigon property, we eventually arrive in a position where there is an available tetrahedron move sweeping $\partial \delta$ across $h$. By \Cref{lem_nohingemoves}, the area of $S^\flat$ can be reduced by flat isotopy.

If $\delta$ is not innermost, note that any $\flat$-disk containing $h$ and contained in $\delta$ must contain $h$ in the shadow of a kink. Hence by passing to a $\flat$-disk contained in $\delta$ we can assume that $\delta$ contains $h$ in the shadow of a kink and that no $\flat$-disks contained in $\delta$ contain $h$. Now we shrink $\delta$ as before, checking after each tetrahedron move to see if we have preserved the bigon property. In shrinking $\delta$, it is possible that $\flat$-disks of the opposite coorientation contained in $\delta$ may grow. If some such disk grows so that it contains $h$, then we are done by \Cref{lem_nohingemoves}. Otherwise we can eventually move $\partial \delta$ across $h$ and \Cref{lem_nohingemoves} again finishes the proof.
 \end{proof}

\begin{lemma}\label{lem_strongbigon}
Let $S^\flat$ be a flattened surface with the bigon property. If $S^\flat$ does not have the efficient bigon property, then $S^\flat$ is flat isotopic to a surface with smaller area.
\end{lemma}

\begin{figure}
\centering
\includegraphics[]{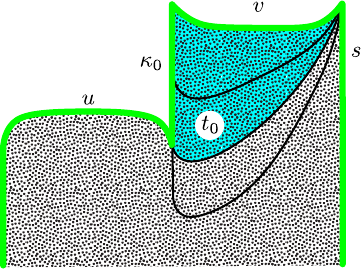}
\caption{Part of $\delta_0$ in the proof of \Cref{lem_strongbigon}.}
\label{fig_efficient3}
\end{figure}

\begin{figure}
\centering
\includegraphics[]{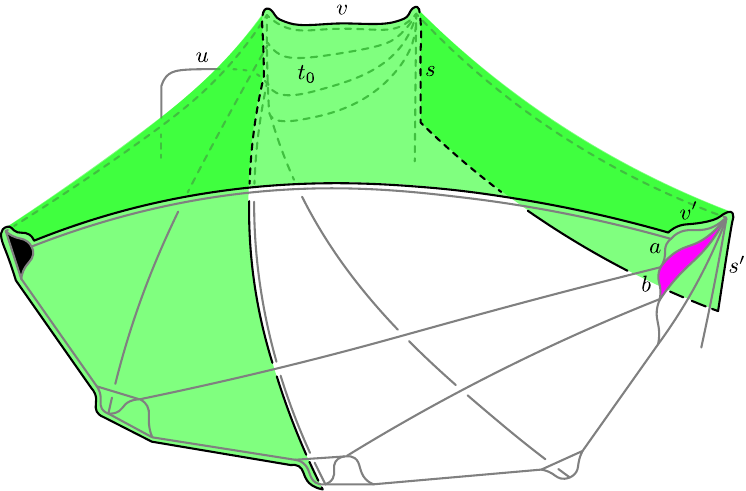}
\caption{Picture of the situation in the proof of \Cref{lem_strongbigon}. The pink flat triangle is $t_0'$.}
\label{fig_efficient34}
\end{figure}

\begin{figure}
\centering
\includegraphics[]{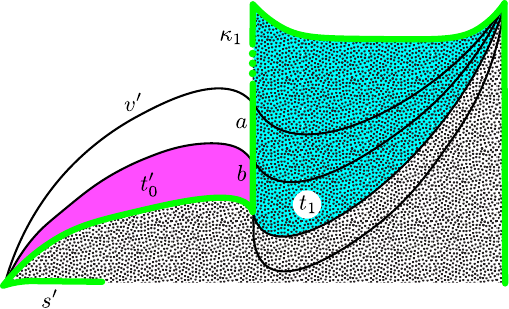}
\caption{Picture of part of $\delta_1$ in the proof of \Cref{lem_strongbigon} after eliminating $\kappa_0$, viewed from inside the manifold. The kink of $\delta_1$ shown is $\kappa_1$. The break in $\kappa_1$ is meant to indicate that the length of $\kappa_1$ is at least 2 and possibly more. This picture and its notation should be compared with \Cref{fig_efficient34}.}
\label{fig_efficient5}
\end{figure}

\begin{proof}
Suppose $\delta_0$ is an inefficient $\flat$-disk of $S^\flat$ with kink $\kappa_0$. By \Cref{lem_shadow}, we can assume $\shad(\kappa_0)$ contains only non-hinge flat triangles.

Let $u, v, s, L_u, L_v$ be defined as in the previous lemma, where $\delta=\delta_0$. We assume that $L_u$ is an upward ladder and $L_v$ is downward; the argument in the other case will be symmetric. 
We may assume the downward flat triangle in $L_v$ immediately below $\shad(\kappa_0)$ is non-hinge. For suppose otherwise; then after a sequence of tetrahedron moves there would be an available tetrahedron move across the hinge tetrahedron and we would be finished by \Cref{lem_nohingemoves}. We have drawn the situation in \Cref{fig_efficient3}, in the case where $\kappa$ has length 2.
Let $t_0$ be the lowest flat triangle in $\shad(\kappa)$, and let $\hex_0$ be the $\tau$-face corresponding to the top of $t_0$.

In \Cref{fig_efficient34}, we have drawn the tetrahedra corresponding to the three flat triangles shown in \Cref{fig_efficient3}. Let $a, b, v', s'$ be the conduits of 
$\del N_\epsilon$
as labeled in \Cref{fig_efficient34}, and let $t_0'$ be the tip colored pink in that figure. Note that $s'$ is a rung conduit since it is a 0-0 edge of an upward flat triangle. Also, since all the flat triangles in $\shad(\kappa)$ are non-hinge, $v'$ is a rung conduit. We conclude that the junction of $v'$ and $s'$ is a cusp of a $\flat$-disk $\delta_1$, by \Cref{lem_negativerungs}.
We may assume that the next conduit of $\partial\delta_1$ in the sequence $(s', v',\dots)$ is not $a$. If it were, then $S^\flat\cap \partial \mr M$ would wrap around 3 conduits of the leftmost flat triangle in \Cref{fig_efficient34} (colored black), and after a tetrahedron move $\partial \mr S^\flat$ would have backtracking so a face move would decrease the area of $S^\flat$.

For clarity, we first assume that $\delta_0$ is innermost. Note that innermostness of $\delta_0$ forces $\delta_1$ to also be innermost, since any $\flat$-disk contained in $\delta_1$ would have to avoid the ladder containing $t_0'$ and thus be width 1, which is impossible since $S^\flat$ has the bigon property.
We can perform tetrahedron moves to shrink $\shad(\kappa_0)$ until $\kappa_0$ no longer exists.  
These moves either introduce a new kink to $\delta_1$, or lengthen a kink on $\delta_1$ (by 2, in the case shown in \Cref{fig_efficient34}). We denote this next kink by $\kappa_1$. 

By \Cref{lem_shadow} we can assume $\shad(\kappa_1)$ consists only of non-hinge flat triangles, and by the same reasoning as above we can assume that the flat triangle immediately below $\shad(\kappa_1)$ is also non-hinge. We denote the bottommost flat triangle in $\shad(\kappa_1)$ by $t_1$; in \Cref{fig_efficient34} $t_1$ is the flat triangle into which the coorientation along $b$ points. We have also drawn a picture of $\delta_1$, using all this notation, in \Cref{fig_efficient5}. We let $\hex_1$ be the $\wt\tau$-face corresponding to the $\partial \hbs$-branch atop $t_1$. Note that $\hex_0$ is the $\tau$-face corresponding to the 0-$\pi$ edge of $t_0'$ which is a rung (see \Cref{fig_efficient34}), so a curve in $\hbs$ starting in $\hex_0$ and moving to $\hex_1$ passes from a bottommost face to a topmost face at $\hex_0\cap \hex_1$ (see \Cref{fig_efficient5}).

Now we may argue for $\delta_1$ exactly as we did for $\delta_0$, either reducing the area of $S^\flat$ by flat isotopy or producing another $\flat$-disk $\delta_2$ with kink $\kappa_2$. If no flat triangle in $\shad(\kappa_2)$ is hinge, we define $t_2$ and $\hex_2$ in the same way as $t_1$ and $\hex_1$ and run the argument again. Assume that we have run the argument $n$ times, and note that as above a curve $\gamma$ in $\hbs$ traversing the $\hex_0,\hex_1,...,\hex_n$ in that order always passes from a lowermost $\tau$-face to an uppermost $\tau$-face when traversing $\hex_i\cap \hex_{i+1}$. By \Cref{lem_branchlines}, the path in the dual graph $\Gamma$ corresponding pushdown of $\gamma$ is a stable branch line. By \Cref{lem_branchlinehinge} some $t_i$ will eventually be hinge, at which point we can apply \Cref{lem_shadow} to reduce the area of $S^\flat$.

Now suppose that $\delta_0$ is non-innermost. By passing to a $\flat$-disk contained in $\delta_0$ and relabeling, we can assume $\delta_0$ is innermost in $\shad(\kappa_0)$, meaning that $\delta_0$ intersects $\shad(\kappa_0)$ and does not contain any $\flat$-disks intersecting $\shad(\kappa_0)$. Note that even if we pass to an inner $\flat$-disk, we do not relabel $\kappa_0$ or $t_0$.

We shrink $\delta_0$ by tetrahedron moves until the first time that either $\delta_0$ or a $\flat$-disk contained in $\delta_0$ has an free size 1 cusp at $t_0$. Then as in the innermost case above, after performing a tetrahedron move on this cusp we have produced an inefficient $\flat$-disk $\delta_1$ with a kink $\kappa_1$. As before we may assume that $\shad(\kappa_1)$ contains no hinge flat triangles and that the flat triangle immediately below $\shad(\kappa_1)$ is non-hinge. We again define $t_1$ to be the lowest flat triangle in $\shad(\kappa_1)$. We now iterate this process. As before, since branch lines always pass through hinge tetrahedra we eventually reach an inefficient $\delta_i$ with $\shad(\kappa_i)$ containing a hinge flat triangle, at which point we are finally done. 
\end{proof}

%%%%%%%%%%%%%%%%%%%%%%%%%%%%%%%%%%%%%%%%%%%%%%%%%%%%%%%%%%%%%%%%%%%%%%%%%%%%%%%%%%%%%%%%%%%%%%%%%%%%%%%%%%%%%%%%%%

\section{Finishing the proof of Theorems \ref{thma} and \ref{thmb}}\label{sec_negregion}

\subsection{Efficient $\flat$-disks}

Let $S^\flat$ be a flattened surface.
Let $r$ be a mixed rod of $S^\flat$ which is incident to a positive plate $P$ and a negative plate $Q$ and corresponds to a $\tau$-edge $e$. If $P$ lies above $Q$ at $e$ with respect to the coorientation on $\hbs$ then we say $S^\flat$ \textbf{curls down at $\boldsymbol{r}$} and otherwise we say $S^\flat$ \textbf{curls up at $\boldsymbol{r}$}. 
\begin{lemma}\label{lem_curling}
Let $S^\flat$ be a flattened surface such that $\partial \mr S^\flat$ has no backtracking, and let $Q$ be a negative plate of $S^\flat$ incident to a rod $r$. If the switch of $Q\cap B^s$ points toward $r$, then $S^\flat$ does not curl down at $r$. If the switch of $P\cap B^u$ points toward $r$, then $S^\flat$ does not curl up at $r$.
\end{lemma}
\begin{proof}
Let $\hex$ be the $\tau$-face corresponding to $Q$, and let $e$ be the $\tau$-edge corresponding to $r$. If the switch of $P\cap B^s$ points toward $r$, then $e$ is the bottom $\tau$-edge of the $\tau$-tetrahedron for which $\hex$ is a bottom $\tau$-face, so $S^\flat$ cannot curl down at $r$ without backtracking. The other case is symmetric.
\end{proof}

We now develop some terminology to discuss efficient $\flat$-disks. An efficient $\flat$-disk $\delta$ has 2 cusps whose complement in $\del\delta$ has two components which we will distinguish in two different ways: top/bottom and positive/negative. The \textbf{top (bottom)} of $\delta$ is the component of $\partial\delta$ along which the coorientation of $\del\tau$ points out of (into) $\delta$.

One of these pieces contains only negative conduit segments and the other contains only positive conduit segments. We call the piece containing the negative conduit segments the \textbf{negative boundary} of $\delta$, denoted $\del\delta_-$. 
Note that the bottom of an outward $\flat$-disk is negative, and the top of an inward $\flat$-disk is negative.

Using this new language, we record a lemma describing the curling of $S^\flat$ at the boundary of $\Neg(S^\flat)$, which as a reminder is the union of all negative plates and rods of $S^\flat$.

\begin{lemma}\label{lem_consistentcurling}
Let $S^\flat$ be a surface with the efficient bigon property. Let $\gamma$ be a component of $\Neg(S^\flat)\cap \partial \mr M$ which is homeomorphic to an interval and contains a rung conduit segment. Let $r_1$ and $r_2$ be the rods of $S^\flat$ corresponding to the junctions incident to the endpoints of $\gamma$. Then $S^\flat$ either curls up at both $r_i$ or down at both $r_i$.
\end{lemma}

\begin{proof}
Because $S^\flat$ has the efficient bigon property, $\gamma$ is equal to the negative boundary  of some $\flat$-disk $\delta$. If $\delta$ is an outward $\flat$-disk, then $\gamma$ is equal to the bottom of $\delta$ and $S^\flat$ must curl downward at both $r_i$. Symmetrically, if $\delta$ is inward then $\gamma$ is the top of $\delta$ and $S^\flat$ must curl upward at both $r_i$.
\end{proof}

\subsection{The excellent bigon property}

Let $S^\flat$ be a flattened surface with with the efficient bigon property. 
Note that $\Neg(S^\flat)$ naturally has the structure of a surface with corners. There are 2 corners for each interval component of $\Neg(S^\flat)\cap \partial U$, and we can smooth the boundary of $\Neg(S^\flat)$ by extending $\Neg(S^\flat)$ slightly into $U$ between each pair of corners as in \Cref{fig_rounding}.
Denote this smoothing of $\Neg(S^\flat)$ by $N$. The reason for forming $N$ is that we will be interested in the indices of patches of certain train tracks on $\Neg(S^\flat)$ where we ignore the corners of $\Neg(S^\flat)$ but include the corners corresponding to train track stops.

\begin{figure}
\centering
\includegraphics[]{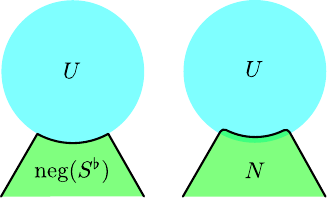}
\caption{$N$ is formed by rounding the corners of $\Neg(S^\flat)\cap \partial U$.}
\label{fig_rounding}
\end{figure}

\begin{lemma}\label{lem_nonegindex}
No patch of $N\cap B^s$ or $N\cap B^u$ has negative index. If $p$ is an annulus patch of $N\cap B^s$ or $N\cap B^u$, then $p$ has no switches and $p\cap \partial \mr S^\flat$ contains only ladderpole conduit segments.
\end{lemma}

\begin{proof}
Let $N^s= N\cap B^s$. Each component of $N\cap \del \mr M$ is either a circle or interval. If $c$ is such a component, then the fact that each $\tau$-face contains a single switch of $\hbs\cap B^s$ onto which it deformation retracts implies that the patch of $N^s$ containing $c$ is homeomorphic to $c\times [0,1]$. Thus, each patch of $N^s$ is either an annulus or a disk. A patch of $N^s$ has corners at exactly those points corresponding to intersections of $N^s$ with $\del N$, i.e. the stops of $N^s$. An annulus patch therefore contains no corners, and a disk patch contains two corners.

Suppose some patch $p$ of $N^s$ has negative index; as a first case suppose that $p$ is a topological disk with two corners. Then $p$ must have at least two switches, whence $p\cap \partial \mr S^\flat$ is a curve containing at least two negative rung conduit segments traversing upward ladders. By \Cref{lem_negativerungs}, $p\cap \partial \mr S^\flat$ is part of the boundary of a $\flat$-disk $\delta$. Note that $\partial \delta$ must have at least two positive rung conduit segments crossing upward ladders, forcing the corresponding nonmeridional disk patch of $S^\flat \cap B^s$ to have at least four switches, contradicting the stable bigon property.

Suppose  now that $p$ is a topological annulus. Then $p\cap \partial \mr S^\flat$ is a closed curve in $\partial N_\epsilon$ containing only negative conduits. A component of $\partial \mr S^\flat$ contains a negative rung conduit segment iff it is the boundary of a $\flat$-disk by \Cref{lem_negativerungs}, and the boundary of a $\flat$-disk must contain positive conduit segments. Hence $p\cap \partial \mr S^\flat$ must contain only ladderpole conduits, so $p$ has no switches by \Cref{lem_tipcondition}.

The same analysis holds symmetrically for $N\cap B^u$.
\end{proof}

\begin{lemma}\label{lem_onlyannuli}
Each component of $N$ is an annulus. 
\end{lemma}

\begin{proof}
We show that no component of $N$ is a disk. It then follows from \Cref{lem_nonegindex} that each component of $N$ is an annulus.

Suppose for a contradiction that $D$ is a disk component of $N$. Let $D^s=D\cap B^s$ and $D^u=D\cap B^u$. Note that each patch of $D^s$ or $D^u$ is topologically a disk and has two corners, so has index equal to $1-\#\{\text{switches of the patch}\}$.

Within this proof we call a patch of $D^s$ or $D^u$ \textbf{small} if it intersects a single negative plate of $S^\flat$. If a patch is not small we will call it \textbf{large}. Let $p$ be a small patch of $D^s$. If $p$ has a switch, then $p\cap \partial \mr M$ is a negative conduit in the boundary of a $\flat$-disk $\delta$. By the efficient bigon property $\partial \delta$ must have at least two negative conduits and they must be unlinked from the positive conduits in $\partial \delta$. This is a contradiction since $p$ is small. Hence $p$ must have no switches, and we conclude $\ind(p)=1$. It follows that $D$ must contain more than one negative plate of $S^\flat$.

Note that each patch of $D^s$ can have at most one switch, for a patch of $D^s$ with more than one switch would imply that there is a $\flat$-disk of $S^\flat$ with two negative ladderpole conduit segments, contradicting the the bigon property. Hence 0 and 1 are the only possible indices of patches of $D^s$. We can define a graph $G_D$ in $D$ with a vertex in the middle of each negative plate, and an edge connecting two vertices exactly when those two plates are connected by a negative rod. Since $D$ is a disk, $G_D$ is a connected tree, hence has at least two leaves (vertices of valence 1). Each leaf must lie interior to a negative plate which abuts two mixed rods and one negative rod. It follows that each of these plates intersects a small patch of $D^s$ and hence $D^s$ has at least two small patches. Since each small patch has index 1 by the above, and the sum of all the indices of $D^s$ is 2, $D^s$ must have exactly two small patches with index 1 and all other patches must have index 0. The same analysis holds for $D^u$. Hence there are two components of $D\cap \partial \mr M$ consisting of a single conduit segment. Call these $\gamma_1$ and $\gamma_2$. There are two components of $\del(D\cap \Neg(S^\flat))-
(\gamma_1\cup\gamma_2)$ (note that $D\cap \Neg(S^\flat)$ is just $D$ before rounding), which we call $A$ and $B$. The situation is as in \Cref{fig_diskexample}.
\begin{figure}
\centering
\includegraphics[]{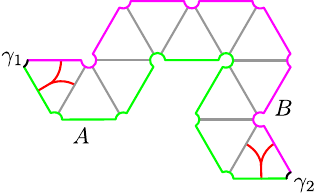}
\caption{$D\cap \Neg(S^\flat)$ in the argument for \Cref{lem_onlyannuli}.}
\label{fig_diskexample}
\end{figure}

Let $\hex_i$ be the plate of $D$ incident to $\gamma_i$ for $i=1,2$. Note that the switch of $\hex_1\cap B^s$ must point toward either $A$ or $B$ since the patch incident to $\gamma_1$ has index 1; suppose without loss of generality that it points toward $A$. This forces the switch of $\hex_1\cap B^u$ to point toward $B$. By \Cref{lem_curling}, $S^\flat$ must curl up at $\hex_1\cap A$ and down at $\hex_1\cap B$.

Applying \Cref{lem_consistentcurling} repeatedly, we see that the curling of $S^\flat$ must be consistent at all rods meeting $A$ and at all rods meeting $B$. We conclude that no switches of $D^s$ point toward $B$ and in particular that the switch of $\hex_2\cap B^s$ points toward $A$, as in \Cref{fig_diskexample}. 

It follows that we may delete all branches of $D^s$ which meet $B-(\hex_1\cup \hex_2)$ and what remains will be another train track $(D^s)'$. By the above analysis there is a unique patch $p'$ of $(D^s)'$ meeting $B-(\hex_1\cup \hex_2)$ and it has two corners and at least two switches, hence has index $\le -1$. This forces some patch of $D^s$ to have negative index, a contradiction. We conclude that no component of $N$ is a disk.
\end{proof}

\begin{lemma}\label{lem_minimalannuli}
Each plate in $\Neg(S^\flat)$ is incident to a mixed rod.
\end{lemma}
\begin{proof}
Let $A$ be a component of $\Neg(S^\flat)$. By \Cref{lem_onlyannuli}, $A$ is an annulus. Suppose there is a proper subset $Y$ of the set of plates in $A$ such that the union $A'$ of $Y$ with all rods incident to two plates in $Y$ is an annulus. Let $A''$ be the annulus obtained by rounding the corners of $A'$. Let $p''$ be a patch of $A''\cap B^s$. Then $p''$ is contained in some patch $p$ of $A^s$ and $\ind(p'')\ge \ind(p)$ since the number of corners is equal for both patches and $p''$ has at most as many switches as $p$. The same holds for patches of $A''\cap B^u$, so we conclude that each patch of $A''\cap B^s$ and $A''\cap B^u$ has nonnegative index. Since $\chi(A'')=0$, the index of each patch must be 0. 

 Because $Y$ does not contain all the plates in $A$ by assumption, there must be some negative rod $r$ in $A$ incident to a plate $P\in A-Y$. But because the index of any patch of $A'\cap B^s$ or $A'\cap B^u$ is 0, each patch touching $r$ has a switch. Hence either $(A'\cup r\cup P)\cap B^s$ or $(A'\cup r\cup P)\cap B^u$ has a patch with two switches, contradicting \Cref{lem_nonegindex}.
 
It follows that there is no such set $Y$ and the claim is proved.
\end{proof}

In light of \Cref{lem_onlyannuli} and \Cref{lem_minimalannuli}, there are two types of possible components of $\Neg(S^\flat)$. If $A$ is a component of $\Neg(S^\flat)$ such that some component of $A\cap \partial \mr M$ is a circle, we say $A$ is a \textbf{ladderpole component of $\boldsymbol{\Neg(S^\flat)}$}. Otherwise, $A$ is a \textbf{non-ladderpole component of $\boldsymbol{\Neg(S^\flat)}$} (see \Cref{fig_negativecomponents}). Ladderpole components are so named because if $A$ is a ladderpole component and $c$ is the component of $\partial A$ lying in $\partial N_\epsilon$, then $c$ has ladderpole slope and consists of only ladderpole conduits by \Cref{lem_nonegindex}.

\begin{figure}
\centering
\includegraphics[height=1.5in]{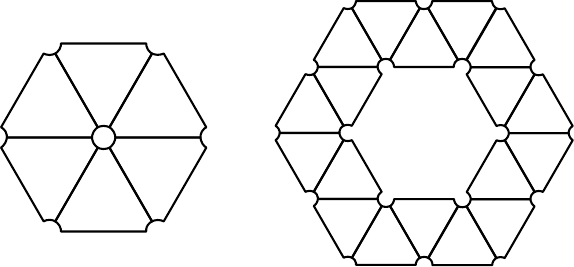}
\caption{Two types of components of $\Neg(S^\flat)$: ladderpole (left) and non-ladderpole (right). Note that each plate has an edge on $\partial N$. As a visual shorthand, we will draw pictures of these components omitting the rods in the plate and rod decomposition, simply drawing them as edges.}
\label{fig_negativecomponents}
\end{figure}

By \Cref{lem_consistentcurling} if $A$ is a component of $\Neg(S^\flat)$ and $C$ is a component of $\partial A$, then the curling of $S^\flat$ is consistent at each mixed rod incident to $C$.

\begin{lemma}\label{lem_nolargebranches}
Let $A$ be a non-ladderpole component of $\Neg(S^\flat)$. Let $A^s: =A\cap B^s$ and $A^u:=A\cap B^u$. If either $A^s$ or $A^u$ contains a large branch then the area of $S^\flat$ can be reduced by flat isotopy.
\end{lemma}

\begin{proof}
Suppose that $A^s$ contains a large branch. Then after a finite sequence of upward flip moves on $S^\flat$, we can replace $A$ by a new annulus $A_+$ with $\partial A_+=\partial A$ with the property that $A_+^s$ does not contain any large branches. After each flip move the switches of the modified $A^s$ are in natural bijection with the switches before the flip. Further, after the flip each switch either points the same way around $A$ as it did before, or points toward $\partial A$. Since $A_+^s$ has no large branches, there must at least one $A_+^s$-switch pointing toward $\del A_+$.

We claim furthermore that there is at least one $A_+^s$-switch pointing toward each component of $\partial A_+$. To see this, let $\del_1$ and $\del_2$ be the 2 boundary components of $A_+$ and suppose that no $A_+^s$-switches point toward $\del _2$. Then, if we delete all branches of $A_+^s$ which touch $\del_2$, we obtain a train track $t$ with a single topological annulus patch $p$ containing $\del_2$. Since there is a positive number of $t$-switches pointing toward $\del_1$, the patch $p$ must have negative index. It follows that some patch of $A_+^s$ must have negative index, which means that the corresponding patch of $A^s$ has negative index. This contradicts \Cref{lem_nonegindex}.

By \Cref{lem_curling}, the curling of $S^\flat$ must be consistent along each component of $S^\flat$ if $\del\mr S^\flat$ has no backtracking. Thus we have shown that either $S^\flat$ curls up at both components of $\partial A$ or $\del\mr S^\flat$ has backtracking, in which case a face move reduces the area of $S^\flat$. Hence we can assume that $S^\flat$ curls up at $\partial A$.
Since $A_+$ was obtained by flipping $A$ upwards, $A_+^u=A_+\cap B^u$ must have a large branch and we can flip $A_+^u$ downward to an annulus $A_-$ with $\partial A_-=\partial A=\partial A_+$ such that $A_-^u$ has no large branches and has switches pointing toward both components of $\partial A_-$. By \Cref{lem_curling}, either $S^\flat$ curls \emph{down} at $\partial A$ or $\del\mr S^\flat$ has backtracking. Since we assumed $S^\flat$ curls up at $\partial A$, we conclude the area of $S^\flat$ can be reduced. A symmetric argument proves the claim for $A^u$.
\end{proof}

\begin{lemma}\label{lem_producingloops}
Let $A$ be a non-ladderpole component of $\Neg(S^\flat)$. If neither $A^s$ nor $A^u$ contains a large branch then one of $A^s$ or $A^u$ carries the core curve of $A$, denoted $\core(A)$, and the image of $\core(A)$ under the collapsing map $\coll\colon N_\epsilon\to \hbs$ is a stable or unstable loop.
\end{lemma}

\begin{proof}
Suppose that $A^s$ does not carry the core of $A$. Note that the endpoints of any line segment $\gamma$ properly embedded in $A$ and carried by $A^s$ must lie on different components of $\partial A$. Otherwise, letting $A'$ denote the component of $N$ containing $A$, one component of $A'-\gamma$ would have index 1, forcing some patch of $A'\cap B^s$ to have positive index.

To show that $A^u$ carries the core of $A$, we will show that $A^u$ has no switches which point toward $\partial A$. Let $P$ be a plate in $A$ incident to negative rods $r_1$ and $r_2$, and a mixed rod $r_3$. Let $e_i$ be the $\tau$-edge corresponding to $r_i$ for $i=1,2,3$. For convenience we can choose the labeling $r_1, r_2, r_3$ to be oriented counterclockwise. Let $p^u$ and $p^s$ be the switches of $A^u$ and $A^s$ lying in $P$. If $p^s$ points toward $\partial A$, then $p^u$ must point toward $r_1$ or $r_2$ by \Cref{lem_traintrackfacts}. Now suppose that $p^s$ points toward $r_1$ or $r_2$; we can assume without loss of generality that $i=2$, which by our labeling forces $r_1$ to be right veering and $r_3$ to be left veering. Let $Q$ be the plate in $A$ which is also incident to $r_2$, and let $q^u$ and $q^s$ be the switches of $A^u$ and $A^s$ respectively lying in $Q$. Label the other rods incident to $Q$ by $r_4$ and $r_5$ so that the triple $(r_2, r_4, r_5)$ is oriented counterclockwise, as in the following picture:
\begin{center}
\includegraphics[]{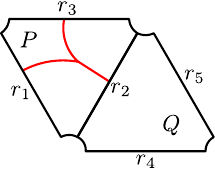}
\end{center}
Suppose that $r_2$ is left veering. Then $q^s$ must point toward $r_5$ since $A^s$ has no large branches. If $r_5$ is a mixed rod then $A^s$ must have a patch of index 0, a contradiction. If $r_5$ is negative then $r_4$ must be mixed, so since $A^s$ does not carry the core of $A$ there must be a properly embedded line segment starting at either $r_3$ or $r_4$ with both its endpoints on the same boundary component of $A$, a contradiction. We conclude that $r_2$ must be right veering, forcing the large half branch incident to $p^u$ to end at $r_1$ by \Cref{lem_traintrackfacts}. Since $A^u$ has no switches pointing toward $\partial A$, it carries $\core(A)$. Since $A^u$ has no large branches, $\core(A)$ collapses to an unstable loop. Symmetric reasoning shows that if $A^u$ does not carry the core of $A$ then $A^s$ does, and that the core then collapses to a stable loop.
\end{proof}

If $A$ is a non-ladderpole component and $A^s$ carries $\core(A)$, which then necessarily collapses to a stable loop, then $A$ is called a \textbf{stable component}. The definition of \textbf{unstable component} is symmetric. We remark that if $A$ is a ladderpole component then  $\core(A)$ collapses to a curve which is both a stable and unstable loop, and so it might be more precise to call stable and unstable components ``strictly stable" and ``strictly unstable," but we forgo this precision.

\begin{definition}
Let $S^\flat$ be a flattened surface with the efficient bigon property. If in addition each component of $\Neg(S^\flat)$ is either stable, unstable, or ladderpole, we say that $S^\flat$ has the \textbf{excellent bigon property}.
\end{definition}

\begin{remark}Suppose that $S^\flat$ has the excellent bigon property. Because $\partial \mr S^\flat$ has no backtracking, $S^\flat$ curls down along stable components of $\Neg(S^\flat)$ and up along the boundaries of unstable components. If $A$ is a ladderpole component, $S^\flat$ could possibly curl up or down along $A$. 
\end{remark}

We compile the results of this subsection by recording the following useful lemma.

\begin{lemma}\label{lem_loops}
Let $S^\flat$ be a surface with the efficient bigon property. Then after a flat isotopy which does not increase area, $S^\flat$ has the excellent bigon property.
\end{lemma}

\subsection{Leveraging the excellent bigon property}\label{sec_leverage}

\begin{lemma}
Let $S^\flat$ be a surface with the excellent bigon property, and let $\delta$ be a $\flat$-disk of $S^\flat$.  Then $\del\delta_-$ contains exactly two rung segments $r_1$ and $r_2$, which satisfy the following: 
\begin{enumerate}[label=(\roman*)]
\item $r_1$ and $r_2$ are adjacent in the sense that they are incident to the same negative junction segment, and
\item at least one of the $r_i$ is incident to a cusp of $\delta$.
\end{enumerate}
\end{lemma}

\begin{proof}
The fact that $\del\delta_-$ contains exactly two rung conduit segments follows from the efficient bigon property. Since $S^\flat$ has the excellent bigon property, $\del\delta_-$ is part of the boundary of a component $A$ of $N$ containing a stable or unstable loop and neither $A^s=A\cap B^s$ nor $A^u=A\cap B^u$ has a large branch.  As such, $\del\delta_-$ never traverses two conduits whose union corresponds to both $0$-$\pi$ edges of any flat triangle.

We shall prove (i) in the case when $\delta$ is outward as the inward case is symmetric.
Suppose that there is a segment of $\del\delta_-$ whose conduit segments, in order, are $r_1, \ell_1,\dots, \ell_n, r_2$ where $r_1$ crosses an upward ladder, $r_2$ crosses a downward ladder, and the $\ell_i$ are ladderpole conduit segments. Let $j$ be the junction of $\ell_n$ and $r_2$, let $s(j)$ be the corresponding $\partial\hbs$-switch, and let $b(\ell_n)$ and $b(r_2)$ be the $\partial\hbs$-branches corresponding to $\ell_n$ and $r_2$ respectively. By the observation in the first paragraph of this proof, $b(r_2)$ must not be the topmost branch at $s(j)$. However, this means that $b(r_2)$ must be a topmost branch at the switch corresponding to  its other endpoint. Since $\del\delta_-$ is the bottom of $\delta$ we see that there must be another negative conduit segment $\ell_{n+1}$ in $\del\delta_-$ after $r_2$. If $t$ is the flat triangle immediately above $b(r_2)$, then $r_2$ and $\ell_{n+1}$ correspond to both 0-$\pi$ edges of $t$, a contradiction. This proves (i).

To see the truth of (ii), we again assume $\delta$ is outward and use the same notation as above, where $r_1$ crosses an upward ladder and $r_2$ crosses a downward ladder. Then $r_1$ must be incident to a cusp of $\delta$ because otherwise there would be backtracking in $\partial \mr S^\flat$. Again, the inward case is symmetric.
\end{proof}

If $\delta$ is a $\flat$-disk of $S^\flat$ with the excellent bigon property and $\del\delta_-$ contains any ladderpole conduit segments in addition to the two adjacent rung conduit segments described in the previous lemma, we say $\delta$ has \textbf{long negative boundary}. Otherwise $\delta$ has \textbf{short negative boundary} (see \Cref{fig_efficientdisk}).
We will orient $\del\delta_-$ so that it starts at a rung conduit and ends at a ladderpole conduit. With this orientation, we call the first conduit of $\del\delta_-$ the \textbf{initial conduit} of $\del\delta_-$. All other conduits are called \textbf{noninitial conduits}. The first negative junction segment traversed by $\del\delta_-$ is called its initial junction and all other negative conduit segments are called \textbf{noninitial junctions}. Note that $\del\delta_-$ has noninitial junctions if and only if it is long.
\begin{figure}
\centering
\includegraphics[height=2in]{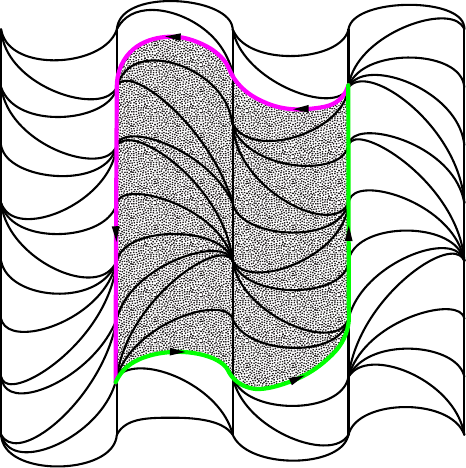}
\caption{An efficient $\flat$-disk $\delta$. If $\delta$ is outward, then the green segment is the negative boundary and it is oriented as shown. If $\delta$ is inward, the magenta segment is the negative boundary and it is oriented as shown. In both cases the negative boundary is long.}
\label{fig_efficientdisk}
\end{figure}

\begin{figure}
\centering
\includegraphics[]{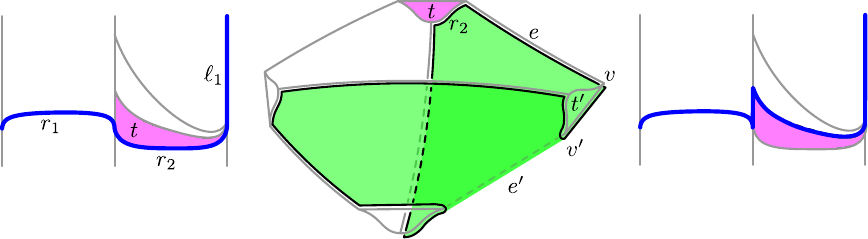}
\caption[]{Pictures in the proof of \Cref{lem_hinge}, left to right: if $t$ is non-hinge, there is an available tetrahedron move across 
$\tet$
which creates a kink in $\partial \delta$.
}
\label{fig_hingelemma}
\end{figure}

\begin{lemma}[noninitial rung lemma]\label{lem_hinge}
Let $S^\flat$ be a surface with the excellent bigon property. Let $\delta$ be a $\flat$-disk of $S^\flat$ with long negative boundary, which we think of as a concatenation of conduits $r_1,r_2,\ell_1,\dots, \ell_n$ where the $r_i$ are rungs and the $\ell_i$ are ladderpoles. If the flat triangle $t$ meeting $r_2$ on the $\delta$ side is non-hinge, then the area of $S^\flat$ can be reduced by flat isotopy.
\end{lemma}

\begin{proof}
We have drawn \Cref{fig_hingelemma} as an aid for reading this proof.

We prove the lemma only in the case when $\delta$ is outward, as the proof in the inward case is symmetric. In this case $\del\delta_-$ is the bottom of $\delta$, so that $t$ is the flat triangle immediately above $r_2$. The fact that $\del\delta_-$ is the bottom of $\delta$ also tells us that if $A$ is the component of $N$ corresponding to $\delta$, then $S^\flat$ must curl down at each mixed rod incident to $A$.

Let $e$ be the $\tau$-edge which meets $\partial \mr M$ at the junction of $r_2$ with $\ell_1$, and let $\tet$ be the tetrahedron corresponding to $t$. Let $t'$ be the tip of $\tet$ on the other side of $e$ from $t$, so that $e$ meets $t'$ at one of its 0-vertices $v$. Note that $t'$ is non-hinge and the fan on the $t'$-side of $v$ is long. If $v'$ is the other 0-vertex of $t'$, the fan on the $t'$-side of $v'$ is short by \Cref{lem_fans}. Note that the edge labeled $e'$ in \Cref{fig_hingelemma} must correspond to a mixed rod of $S^\flat$ since each plate of $\Neg(S^\flat)$ is incident to a mixed rod by \Cref{lem_minimalannuli}.
Since $S^\flat$ curls down along $\partial A$, there is a  tetrahedron move of $S^\flat$ across $\tet$. If $\delta$ is not innermost then this move is possibly of a $\flat$-disk contained in $\delta$. Either way, the move creates a kink in the boundary of some $\flat$-disk.

Now by \Cref{lem_residualbp} and \Cref{lem_strongbigon}, we can perform a flat isotopy to reduce the area of $S^\flat$.
\end{proof}

\begin{lemma}[nonterminal ladderpole lemma]\label{lem_non-hinge}
Let $S^\flat$, $\delta$, $r_1,r_2,\ell_1,\dots, \ell_n$ be as in \Cref{lem_hinge} and suppose $n\ge 2$. If the flat triangle $t$ on the $\delta$-side of $\ell_i$ for $i=1,\dots, n-1$ is hinge, then the area of $S^\flat$ can be reduced by flat isotopy.
\end{lemma}
\begin{figure}
\centering
\includegraphics[]{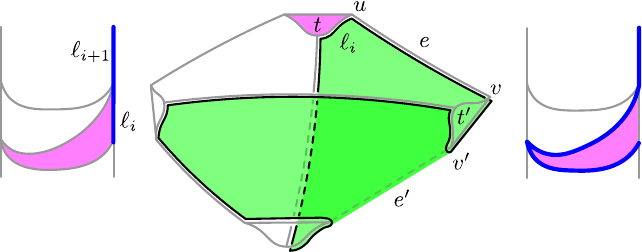}
\caption[]{Pictures in the proof of \Cref{lem_non-hinge}, left to right: if $t$ is hinge, there is an available tetrahedron move across 
$\tet$
, after which $S^\flat$ no longer has the unstable bigon property.}
\label{fig_non-hingelemma}
\end{figure}
\begin{proof}
We have drawn \Cref{fig_non-hingelemma} as an aid for reading this proof.

Again, we will prove the lemma only in the case when $\delta$ is outward, in which case $\del\delta_-$ is the bottom of $\delta$. Let $\tet$ be the tetrahedron corresponding to $\delta$. Suppose for a contradiction that $t$ is hinge. Let $u$ be the $\partial\hbs$-switch corresponding to the junction of $\ell_i$ with $\ell_{i+1}$. Then $t$ is a downward flat triangle which is bottommost in a long fan of $u$. Let $e$ be the $\tau$-edge terminating at $u$, let $v$ be the other endpoint of $e$ on $\partial\mr M$, and let $t'$ be the tip of $\tet$ on the other side of $e$ from $u$, so that $v$ is a 0-vertex of $t'$. Let $v'$ be the other 0-vertex of $t'$. Note that $t'$ is an upward flat triangle and that it is not topmost at $v$, so it is topmost at $v'$. As in \Cref{lem_hinge}, \Cref{lem_minimalannuli} implies that the edge labeled $e'$ in \Cref{fig_non-hingelemma} corresponds to a mixed rod of $S^\flat$. Since $S^\flat$ curls down along $\partial A$, there is a tetrahedron move of $S^\flat$ across $\tet$. As in \Cref{lem_hinge}, this is true regardless of whether $\delta$ is innermost. After this tetrahedron move $S^\flat$ does not have the bigon property. By \Cref{lem_residualbp} we can reduce the area of $S^\flat$ by flat isotopy.
\end{proof}

\begin{lemma}[annulus lemma]\label{lem_annulusmove}
Let $S^\flat$ be a surface with the excellent bigon property. Let $A$ be a stable component of $\Neg(S^\flat)$.  Then either $\core(A)$ collapses to a shallow stable loop or the area of $S^\flat$ can be reduced by flat isotopy. Symmetrically, if $A$ is unstable then either $\core(A)$ collapses to a shallow unstable loop or the area of $S^\flat$ can be reduced by flat isotopy.

In the case of shallow stable or unstable loops, there is an available annulus move of $S^\flat$ after which $\Neg(S^\flat)$ has ladderpole components.
\end{lemma}

\begin{proof}
It follows from \Cref{lem_hinge} and \Cref{lem_non-hinge} that if $\delta$ is an outward $\flat$-disk with long negative boundary then either the area of $S^\flat$ is reducible, or at each noninitial junction, $\del\delta_-$ passes from a second-topmost to a topmost conduit. Symmetrically, if $\delta$ is an inward $\flat$-disk with long negative boundary then either the area of $S^\flat$ is reducible, or at each junction of noninitial segments, $\del\delta_-$ passes from a second-bottommost to a bottommost conduit. 

Similarly to the preceding lemmas, the case when $A$ is stable is illustrative; suppose $A$ is stable. Each patch of $A^s$ intersects $\partial \mr M$ in the negative boundary of a $\flat$-disk. Then for each negative rod $r$ in $A$, there is a component of $r\cap \partial\mr M$ which is a noninitial junction (see \Cref{fig_annuluslemma}). It follows that the core of $A$ collapses to a shallow stable loop, and we may perform an annulus move.
\end{proof}

\begin{figure}
\centering
\includegraphics[]{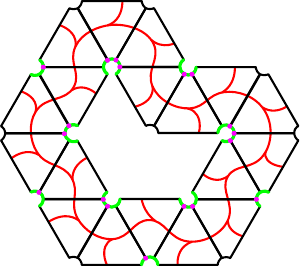}
\caption[]{A stable component of $\Neg(S^\flat)$. We have indicated long negative boundaries by green segments, and noninitial junctions by dots.}
\label{fig_annuluslemma}
\end{figure}

The following lemma is the final piece we need to prove Theorems \ref{thma} and \ref{thmb} from the introduction.

\begin{lemma}[ladderpole component lemma]\label{lem_ladderpole}
Let $S^\flat$ be a surface with the excellent bigon property. Suppose $\Neg(S^\flat)$ has a ladderpole component $A$. Then the area of $S^\flat$ can be reduced by flat isotopy.
\end{lemma}

\begin{proof}
Suppose that $S^\flat$ curls down at $\partial A$; the argument when $S^\flat$ curls up is symmetric.
\begin{figure}
\centering
\includegraphics[]{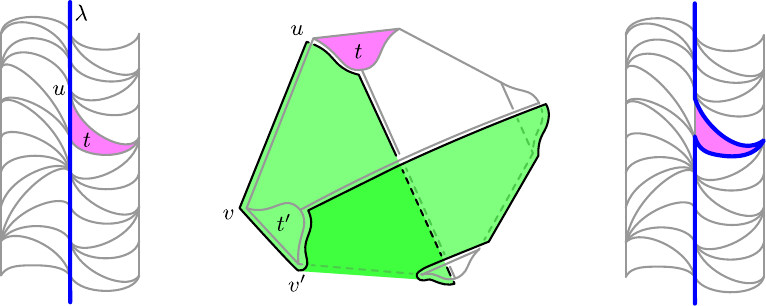}
\caption{Pictures from the proof of \Cref{lem_ladderpole}.}
\label{fig_ladderpolelemma}
\end{figure}
There is a component $U_A$ of $U$ such that $\partial A$ meets $\partial U_A$ in a curve $\lambda$ containing only ladderpole conduit segments. Then there is a hinge downward flat triangle $t$ which is incident to a conduit segment $c$ of $\lambda$. Let $\tet$ be the tetrahedron corresponding to $t$, and let $u$ be the endpoint of $c$ which is higher on $\lambda$. Let $t'$, $v$, and $v'$ be as in \Cref{fig_ladderpolelemma}. Then $t$ is bottommost in its side of $u$, and hence $t'$ is bottommost in its side of $v$ and topmost in its side of $v'$ by \Cref{lem_fans}. Since $S^\flat$ curls down along $\partial A$, there must be an available tetrahedron move across $\tet$. The effect of this move on $\lambda$ shows that $S^\flat$ no longer has the bigon property. By \Cref{lem_residualbp}, $S^\flat$ is flat isotopic to a surface with less area.
\end{proof}

\begin{theorem}\label{thm_mainthmveering}
Let $\tau$ be a veering triangulation of $\mr M$, and let $M$ be obtained from $M$ by Dehn filling along slopes with $\ge 3$ prongs. Let $S$ be a taut surface such that $[S]\in \cone(\sigma_\tau)$. Then $S$ is carried by $\tau^{(2)}$ up to isotopy.
\end{theorem}

\begin{proof}
Let $S^\flat$ be a flattening of $S$ that minimizes area in its flat isotopy class. By \Cref{lem_strongbigon}, $S^\flat$ has the efficient bigon property. 
We wish to show that $\Neg(S^\flat)$ is empty. By \Cref{lem_loops}, $S^\flat$ has the excellent bigon property. By \Cref{lem_annulusmove}, if $\Neg(S^\flat)$ is nonempty then up to a flat isotopy which does not increase area we can assume $\Neg(S^\flat)$ has ladderpole components. However, if $\Neg(S^\flat)$ had a ladderpole component then $S^\flat$ would not minimize area by \Cref{lem_ladderpole}. We conclude that $\Neg(S^\flat)=\varnothing$, so $S^\flat$ is carried by $\tau^{(2)}$. 
 \end{proof}

The above proof is nonconstructive as written, but within it is a recipe for producing the carried surface. We provide this recipe explicitly now:

\begin{enumerate}[label={{(\arabic*)}}, start=0]
\item Given a taut surface $S$ with $[S]\in \cone(\sigma_\tau)$, flatten it and call the resulting surface $S^\flat$. Proceed to step 1.

\item Is $\Neg(S^\flat)$ empty?
\begin{enumerate}
\item If so, then $S^\flat$ is carried by $\hbs$ and we are done.
\item If not, proceed to step 2.
\end{enumerate}

\item Does $\partial \mr S^\flat$ have backtracking?
\begin{enumerate}
\item If so, reduce the area of $S^\flat$ by a face move and return to step 1.
\item If not, proceed to step 3.
\end{enumerate}

\item Does $S^\flat$ have a $\flat$-disk of width $\le 1$?
\begin{enumerate}
\item If so, reduce the area of $S^\flat$ by a width one move and return to step 1.
\item If not, then $S^\flat$ has the bigon property by the proof of \Cref{lem_indexsumarg}. Proceed to step 4.
\end{enumerate}

\item Does $S^\flat$ have the efficient bigon property?
\begin{enumerate}
\item If not, reduce the area of $S^\flat$ using methods from the proof of \Cref{lem_strongbigon} and return to step 1.
\item If so, proceed to step 5.
\end{enumerate}

\item Does $S^\flat$ have the excellent bigon property?
\begin{enumerate}
\item If not, reduce the area of $S^\flat$ using methods from the proof of \Cref{lem_nolargebranches} and return to step 1.
\item If so, proceed to step 6.
\end{enumerate}

\item Does $\Neg(S^\flat)$ have a ladderpole component?
\begin{enumerate}
\item If so, reduce the area of $S^\flat$ using methods from the proof of \Cref{lem_ladderpole} and return to step 1.
\item If not, then the core of every component of $\Neg(S^\flat)$ collapses to a shallow stable or unstable loop. Perform an annulus move, after which $\Neg(S^\flat)$ has a ladderpole component. Reduce the area of $S^\flat$ by \Cref{lem_ladderpole} and return to step 1.
\end{enumerate}
\end{enumerate}

\Cref{thm_mainthmveering} immediately gives the following:

\begin{corollary}\label{cor_reversecontain}
$\cone(\sigma_\tau)\subset\C_\tau^\vee$, so $\cone(\sigma_\tau)=\C_\tau^\vee$. 
\end{corollary}

\begin{proof}
If $\alpha$ is an integral class in $\cone(\sigma_\tau)$ then \Cref{thm_mainthmveering} produces a surface representing $\alpha$ carried by $\hbs$, which clearly pairs nonnegatively with each closed positive transversal to $\hbs$ lying in $\mr M$.
\end{proof}

As a summary of the paper to this point we provide a proof of \Cref{thma} and \Cref{thmb}.

\begin{maintheorem}\label{thm_tautsurfaces}
Let $\tau$ be a veering triangulation of a compact 3-manifold $\mr M$. If $M$ is obtained by Dehn filling each component of $\partial \mr M$ along slopes with $\ge 3$ prongs then $M$ is irreducible and atoroidal. Let $\sigma_\tau$ be the face of the Thurston norm ball $B_x(M)$ determined by the Euler class $e_\tau$. Then the following hold:
\begin{enumerate}[label=(\roman*)]
\item $\cone(\sigma_\tau)=\C_\tau^\vee$, and the codimension of $\sigma_\tau$ in $\partial B_x(M)$ is equal to the dimension of the largest linear subspace contained in $\C_\tau$.
\item If $S\subset M$ is a surface, then $S$ is carried by $\hbs$ up to isotopy if and only if $S$ is taut and $[S]\in \cone(\sigma_\tau)$.
\end{enumerate}
\end{maintheorem}

\begin{proof}
We get irreducibility of $M$ from \Cref{lem_irreducible} and atoroidality by \Cref{lem_atoroidal}.

 The containment $\C_\tau^\vee\subset \cone(\sigma_\tau)$ is given by \Cref{lem_easycontain}. The reverse containment is given by \Cref{cor_reversecontain}. Now the claim about the dimension of $\sigma_\tau$ in statement (i) is basic linear algebra and follows from the fact that $\C_\tau$ and $\C_\tau^\vee$ are dual convex polyhedral cones, see e.g. \cite[\S 1.2]{Ful93}. This completes the proof of (i).
 
If $S$ is carried by $\hbs$ up to isotopy then $S$ is taut and $[S]\in \cone(\sigma_\tau)$ by \Cref{lem_ifcarriedthentaut}. The other direction of statement (ii) is \Cref{thm_mainthmveering}. 
\end{proof}

\section{The case with boundary}\label{sec_export}

Now we explain how to modify our methods to obtain a result in the case when the manifold we care about is the one triangulated by $\tau$, rather than a Dehn filling. The result in this section is used in \cite{LMT20} to show that a veering triangulation determines a face of the Thurston norm ball of $H_2(\mr M,\partial \mr M)$.

We begin by expanding our definition of partial branched surface.
Let $M$ be a compact 3-manifold with boundary. A \textbf{partial branched surface} in $M$ is a branched surface $B$ in $M-\intr(U)$, where $U$ is a union of closed solid tori and closed regular neighborhoods of components of $\partial M$. A properly embedded surface $S\subset M$ is \textbf{carried} by a partial branched surface $B\subset M$ if $S$ has no components completely contained in $U$, $S-\intr(U)$ is carried by $B$, and each component $S\cap U$ $\pi_1$-injects into its component of $U$.

Let $\tau$ be a veering triangulation of a compact manifold $\mr M$. Let $M$ be obtained by gluing a thickened torus $T^2\times [0,1]$ to each boundary component of $\mr M$. Though $M$ and $\mr M$ are homeomorphic, it is important to distinguish between them. Note $\hbs$ is a partial branched surface in $M$ with $U=M-\mr M$. A properly embedded surface $S\subset M$ is carried by $\hbs$ as a partial branched surface if $S\cap \mr M$ is carried by $\hbs$ and a branched surface and each component of $S\cap U$ is an annulus which $\pi_1$-injects into its component of $U$.

We define $N_\epsilon$ just as in \Cref{sec_rodplate}. A \textbf{flattened surface} $S^\flat$ in $M$ is a properly embedded incompressible and $\del$-incompressible surface in $M$ such that $S^\flat \cap \mr M$ lies in $N_\epsilon$ transverse to the vertical foliation and each component of $S^\flat\cap U$ is $\pi_1$-injective in its component of $U$. Terms like plate, rod, and area translate immediately to this setting. A \textbf{$\boldsymbol{\flat}$-disk} of $S^\flat$ is a disk in $\partial \mr M$ bounded by a component of $S^\flat\cap \partial \mr M$. Note that if $\delta$ is a $\flat$-disk, then $\del\delta$ must also bound a disk in $S^\flat\cap U$. Words like inward, outward, volume, and cusp that we previously used to describe $\flat$-disks in the case of a closed 3-manifold also translate immediately.

\begin{theorem}\label{thm_export}
Let $S$ be an incompressible and boundary incompressible surface in $M$, where $M$ is obtained by gluing a thickened torus to each component of $\mr M$. Further suppose that $S$ has the property that for any surface $S'$ isotopic to $S$ that is transverse to $B^u$ and $B^s$, either: 
\begin{itemize}
\item one of $S' \cap B^u$ or $S' \cap B^s$ has a patch of positive index, or
\item every negative switch of $S' \cap B^u$ and every negative switch of $S' \cap B^s$ belongs to a bigon.
\end{itemize}
Then $S$ is isotopic to a surface carried by the partial branched surface $\tau^{(2)} \cap \mr M$.
\end{theorem}

\begin{proof}
We first flatten $S$ to a flattened surface $S^\flat$. This works exactly as in \Cref{sec_flatteningprocess}. Now our definitions of flat isotopies translate verbatim to this setting. Recall that the face move in the closed case used irreducibility of the closed manifold. 
As noted in the proof of \Cref{lem_atoroidal} $M$ is both irreducible and atoroidal, so there is no problem.
Hence we can apply face moves and width 1 moves until $S^\flat$ has no $\flat$-disks of width 1, and thus neither $S^\flat\cap B^s$ nor $S^\flat \cap B^u$ has a patch of positive index. By the condition in the theorem statement, every negative switch of $S^\flat \cap B^s$ belongs to a bigon and similarly for $B^u$. We can now apply our argument from the closed case to reduce the area of $S^\flat$ until it has no negative plates, beginning with \Cref{lem_negativerungs} and proceeding directly through the proof of \Cref{thm_mainthmveering}.
\end{proof}

\bibliographystyle{alpha}
\bibliography{bibliography}
\label{sec:biblio}

\end{document}